\documentclass[12pt]{amsart}
 \usepackage{amsfonts,amssymb}

\usepackage{mathrsfs}
\let\mathcal\mathscr

\usepackage[all,ps,cmtip]{xy} 

\def\Z{{\bf Z}}

\def\F{{\bf F}}
\def\C{{\bf C}}

\def\P{{\bf P}}

\def\ppav{principally polarized abelian variety}

\def\phi{{\varphi}}

\def\cI{\mathcal{I}}

\def\cA{\mathcal{A}}

\def\cF{\mathcal{F}}
\def\cL{\mathcal{L}}
\def\cO{\mathcal{O}}
\def\cG{\mathcal{G}}

\def\cH{\mathcal{H}}
\def\cE{\mathcal{E}}
\def\cS{\mathcal{S}}

\def\cM{\mathcal{M}}
\def\cN{\mathcal{N}}
\def\cK{\mathcal{K}}
\def\cT{\mathcal{T}}

\def\cY{\mathcal{Y}}
\def\cX{\mathcal{X}}

\def\lra{\longrightarrow}
\def\llra{\hbox to 10mm{\rightarrowfill}}
\def\lllra{\hbox to 15mm{\rightarrowfill}}

\def\llla{\hbox to 10mm{\leftarrowfill}}
\def\lllla{\hbox to 15mm{\leftarrowfill}}

\def\dra{\dashrightarrow}
\def\isom{\simeq}
\def\eps{\varepsilon}
\def\ie{\hbox{i.e.}}

\def\Im{\mathop{\rm Im}\nolimits}
\def\vide{\varnothing}

\DeclareMathOperator{\isomto}{\stackrel{{}_{\scriptstyle\sim}}{\to}}

\DeclareMathOperator{\isomlra}{\stackrel{{}_{\scriptstyle\sim}}{\lra}}

\DeclareMathOperator{\Sym}{Sym}

\DeclareMathOperator{\Pic}{Pic}

\DeclareMathOperator{\Hom}{Hom}
\def\Im{\mathop{\rm Im}\nolimits}

\DeclareMathOperator{\Alb}{Alb}

\DeclareMathOperator{\Ext}{Ext}
\DeclareMathOperator{\Id}{Id}
\DeclareMathOperator{\Aut}{Aut}
\DeclareMathOperator{\PGL}{PGL}
\DeclareMathOperator{\GL}{GL}
\DeclareMathOperator{\SL}{SL}

\DeclareMathOperator{\Bir}{Bir}
\DeclareMathOperator{\len}{length}

\def\llra{\hbox to 10mm{\rightarrowfill}}
\def\lllra{\hbox to 15mm{\rightarrowfill}}

\newtheorem{lemm}{Lemma}[section]
\newtheorem{theo}[lemm]{Theorem}
\newtheorem{coro}[lemm]{Corollary}
\newtheorem{prop}[lemm]{Proposition}

\newtheorem*{conj}{Conjecture}
\newtheorem*{prob}{Problem}

\theoremstyle{definition}

\newtheorem{rema}[lemm]{Remark}

\theoremstyle{remark}
\newtheorem*{remark*}{Remark}
\newtheorem*{note*}{Note}

\def\ot{\otimes}
\def\op{\oplus}
\def\ra{\to}

\begin{document}
\title[Prime Fano threefolds  of degree 10]{On the period map for   prime Fano threefolds of degree 10}
\author[O. Debarre]{Olivier Debarre}
\address{D\'epartement de  Math\'ematiques et Applications -- UMR 8553,
\'Ecole Normale Sup\'erieure,
45 rue d'Ulm, 75230 Paris cedex 05, France}
\email{{\tt odebarre@dma.ens.fr}}
\author[A. Iliev]{Atanas Iliev}
\address{Institute of Mathematics,
Bulgarian Academy of Sciences,
Acad. G. Bonchev Str., bl. 8, 
1113 Sofia, Bulgaria}
\email{{\tt ailiev@math.bas.bg}}
\author[L. Manivel]{Laurent Manivel}
\address{Institut Fourier,  
Universit\'e de Grenoble I et CNRS,
BP 74, 38402 Saint-Martin d'H\`eres, France}
\email{{\tt Laurent.Manivel@ujf-grenoble.fr}}

\def\moins{\mathop{\hbox{\vrule height 3pt depth -2pt
width 5pt}\,}}

\maketitle
\tableofcontents

  \section{Introduction: two problems about   Fano threefolds of degree $10$}\label{intro}

There are $17$   families of 
smooth Fano threefolds with Picard number $ 1$: the projective  space $\P^3$, 
the smooth  quadric $Q\subset\P^4$, the five families  $\cY_d$, $d\in\{1,\dots,5\}$ 
of Fano threefolds of index $2$ and degree $d$, and the $10$ families   
$\cX_{2g-2}$, $g\in\{2,3,\dots,10,12\}$ of Fano threefolds of index $1$ and degree $2g-2$ (see table in  \cite{ip}, \S 12.2). We will denote by $Y_d$ a Fano threefold belonging to the family $\cY_d$, and by $X_{2g-2}$ a member of $\cX_{2g-2}$.

Any threefold from the $8$ families $\P^3$, $Q$, $\cY_4$, $\cY_5$, $\cX_{12}$, $\cX_{16}$, $\cX_{18}$, and $\cX_{22}$ 
is rational.

For  threefolds $X$ from the remaining $9$ families  
$\cY_1$, $\cY_2$ (quartic double solids), $\cY_3$  (cubics), and $\cX_{2g-2}$, $g\in\{2,\dots,6,8\}$,
there has been two   approaches to proving or disproving rationality: 
\begin{itemize}
\item studying the  group $\Bir(X)$ 
of birational automorphisms of $X$. This group is known for  $ X_2$, $X_4$, $X_6$, and $Y_1$   (\cite{Is2}, \cite{Grin});
\item studying the intermediate Jacobian $J(X)$ of $X$. This \ppav\ has been well-studied for 
$ Y_2$, $ Y_3$, and $ X_8$. In particular, the  Torelli theorem holds (\cite{cg}, \cite{voi}, \cite{deb}).
\end{itemize}
One important outcome   is that any $X$ from the above 
$7$ families is not rational: in the first case because  
$\Bir(X)$ differs from the Cremona group $\Bir(\P^3)$, and in the second case
 because $J(X)$ is not  a product of Jacobians of curves. 

\medskip
The two remaining families are $\cX_{14}$ and $\cX_{10}$. 

It was known to Fano that any   
$X_{14}$ is birational to a smooth cubic threefold  $Y_3$.  
In particular, no   $X_{14}$ is rational (\cite{bbb}, Theorem 5.6 (i)).
Moreover, this implies that the $5$-dimensional intermediate Jacobians $J(X_{14})$ and $J(Y_3)$ are isomorphic. Together with the Torelli theorem for cubic threefolds, 
this implies that the set of all $X_{14}$ that are birational to a given cubic $Y_3$   
is an irreducible fivefold birational to   $ J(Y_3)$ and that this set is  the fiber through $[X_{14}]$  of the period map   $\cX_{14}\to\cA_5$  
(\cite{IMr}). 
   
\medskip  
 Many geometrical properties of $X_{10}$ were discovered   in 1982 by Logachev  (see \cite{lo}, posted 22 years later), but neither is the group $\Bir(X_{10})$   known, nor has
the   intermediate Jacobian $J(X_{10})$   been   studied.

 As noticed by Tyurin, the Fano threefold $X_{10}$ shares certain   properties 
with the quartic double solid $Y_2$. In particular,    their 
intermediate Jacobians are both $10$-dimensional.\footnote{\label{pgq}Unknown to Tyurin at the time, there are also numerical coincidences between the invariants of the
 so-called Fano surface $F(Y_2)$ of lines on   $Y_2$ (computed by Welters in   \cite{wel}), and those of    (the minimal model of) the Fano surface $F(X_{10})$ of {\em conics} on  $X_{10}$ (computed by Logachev in \cite{lo}): they   both satisfy  $p_g = 101$, $q = 10$, $c_1^2 = 720$, and $c_2 = 384$. This coincidences stem from the fact that a general $Y_2$ is isomorphic to a {\em singular} $X_{10}$ (see Proposition \ref{explain}).}  
 Thinking that the situation would be analogous to the correspondences between   $X_{14}$  and  $Y_3$ described above, Tyurin   stated the following (\cite{Tyu}, p. 739). 
 
\begin{conj}[Tyurin, 1979]\label{tc} 
The general Fano threefold $X_{10}$ is birational to a quartic double solid. 
\end{conj}

We give a negative answer 
to this conjecture (Corollary \ref{notiso}).

\medskip

Very little is known about the group $\Bir(X_{10})$ 
(\cite{Is3}, Problem 4 (b)). 
As far as we know, the only   approach  was initiated 
in the 80's by Khashin: in the short note \cite{has},    he describes birational involutions  of $ X_{10} $ (associated with points, twisted cubics, and elliptic quartics contained in $X_{10}$). Given a line $\ell\subset X_{10}$ (resp. a conic $c\subset X_{10}$), he also constructs a birational isomorphism   $\psi_\ell:X_{10}\dra X_\ell$ (resp. $\psi_c:X_{10}\dra X_c$), where $X_\ell$ (resp. $X_c$) is also in $\cX_{10}$ (see \S\ref{el} and \S\ref{ec}).\footnote{Khashin also stated without proof that the existence of a birational isomorphism  $f$ between  $X_{10}$ and any other Fano threefold $Y$ of index one, should imply that $Y$ is also   in $\cX_{10}$, and that $f$ is a composition of birational isomorphisms of the five types described above.  }

Since $X_\ell$ and $X_c$ have the same intermediate Jacobian as $X_{10}$, the matter of deciding whether they are isomorphic is of great interest for the Torelli problem. 
However, this problem  remained unsolved for years (\cite{ip}, Problem 11.4.2 (ii)).

\begin{prob}[Khashin, 1986]
Are   $X_\ell$ and $X_c$ isomorphic to $X_{10}$?
\end{prob}

We also give a negative answer   to this problem, by 
proving that in  
general, neither $X_\ell$, nor $X_c$, is isomorphic to $X_{10}$. 
\medskip

Both answers are consequences of our study of the period map $\wp:\cX_{10}\to \cA_{10}$.
We show  (Theorem \ref{fibb}) that the set  of all conic transforms $X_c$, as $c$ varies in the Fano surface $F(X_{10})$ of conics contained in $X_{10}$, forms a connected component of the fiber of $\wp$ through $[X_{10}]$ which is  birationally isomorphic to the quotient  of $F(X_{10})$ by a geometrically meaningful involution $\iota$.

On the other hand, consider points of $\cX_{10}$ corresponding to general line transforms $X_\ell$. Their conic transforms form another $2$-dimensional component of the fiber of $\wp$ over $[J(X_{10})]$ which is birationally isomorphic to the quotient by  a geometrically meaningful involution of the moduli space $\cM_X(2;1,5)$ of stable rank-$2$ torsion-free sheaves on $X_{10}$  with Chern numbers $c_1=1$, $c_2=5$, and $c_3=0$, itself a smooth irreducible surface (Proposition \ref{pro81}  and Theorem \ref{th82}).

We conjecture that a general fiber of the period map $\wp$ is just the union of these two surfaces.
  Although the  study of certain degenerations of elements of $\cX_{10}$ provide strong evidence for this conjecture, our lack of  knowledge about the properness of $\wp$ makes it very difficult to
 obtain more information about it at the moment.

 \section{Notation}
 
 $\bullet$ As a general rule, $V_m$ denotes an $m$-dimensional  vector space, and $\Gamma_d^g$ a degree-$d$ curve with geometric genus $g$. Fiber bundles are denoted by script letters; for example, $\cS_{r,V}$ is the rank-$r$ tautological subbundle on the Grassmannian $G(r,V)$.
 
$\bullet$  $V_5$ is a $5$-dimensional complex vector space, $V_{10}=\wedge^2 V_5$, and the Pl\"ucker map embeds $G(2,V_5) $ into $\P_9=\P(V_{10})$.
 
  $\bullet$ $V_8\subset\wedge^2 V_5$ is a   codimension-$2$ linear subspace whose orthogonal      $L\subset \wedge^2 V^\vee_5$ consists of skew forms on $V_5$ that are all of maximal rank $4$, and $\P_7=\P(V_8)$.

 $\bullet$ $W$ is the  {\em smooth} $4$-fold $G(2,V_5)\cap \P_7$,  of degree $5$ in $ \P_9$ (\S\ref{fourw}).  
  
 $\bullet$ $U_3\subset V_5$ is the unique $3$-dimensional subspace totally isotropic for all forms in 
 $V_8^\bot$, with dual line $L_U\subset\P(V_5^\vee)$.
 
 $\bullet$ $\Pi$ is the $2$-plane $G(2,U_3)\subset W$.

 $\bullet$  $\Omega$ is a quadric in $ \P_9$ such that
$X=W\cap \Omega=G(2,V_5)\cap \P_7 \cap \Omega\subset  \P_9
$ is a (smooth) Fano threefold (\S\ref{fanox}).

 $\bullet$ $c_X=\Pi\cap \Omega$ is the unique $\rho$-conic on  $X$ (\S\ref{pla}).

 $\bullet$ $F_g(X)$ is the smooth connected surface that parametrizes conics on $X$ (the letter $g$ stands for ``geometric'') (\S\ref{surF}).
 
 $\bullet$ $L_\sigma\subset F_g(X)$ is the curve of $\sigma$-conics in $X$ (\S\ref{surF}). It is   exceptional, and if $F_g(X)\to F_m(X)$ is its contraction, $F_m(X)$ is a smooth minimal surface of general type.

 $\bullet$ $F(X)=\{(c,V_4) \in F_g(X)\times \P(V_5^\vee) \mid  c\subset G(2,V_4)\}
$ is a smooth surface and $p_1:F(X)\to F_g(X)$ is the blow-up of the point corresponding to $c_X$ (\S\ref{surF}, \S\ref{stbt}).

$\bullet $ $\iota$ is the fixed-point-free involution on $F(X)$ defined in \S\ref{invo}. We set $F_\iota(X)=F(X)/\iota$ and $F_{m,\iota}(X)=F_m(X)/\iota$. There is a diagram
$$
\xymatrix
{F(X)\ar[rr]^{\hbox{\tiny quotient by }\iota}\ar[d]_{\scriptstyle\hbox{\tiny blow-up of }[c_X] }^{p_1}&&F_\iota(X)\ar[dd]^{\scriptstyle\hbox{\tiny blow-up of } \pi([c_X]) }\\
F_g(X)\ar[d]_{\scriptstyle\hbox{\tiny blow-up of } \iota([c_X])}\\
F_m(X)\ar[rr]^{\rm quotient\ by\ \iota}_\pi&&F_{m,\iota}(X)
}
$$ 

$\bullet$ For any hyperplane $ V_4\subset V_5$, we let  $M_{V_4} $ be the $4$-dimensional vector space $ \wedge^2 V_4 \cap V_8$.

$\bullet$ $Q_{W,V_4}$ is the quadric surface $G(2,V_4)\cap \P(M_{V_4})$, contained in $W$.

$\bullet$ $Q_{\Omega,V_4}$ is the quadric surface $\Omega\cap \P(M_{V_4})$.

\section{The fourfold $W$}\label{sec3}

Except for the cohomology calculations of \S\ref{coca}, all the material in this section is either classical, or is due to Logachev (\cite{lo}).

\subsection{Lines, $2$-planes, and conics   in $G(2,V_5)$}\label{pla} We denote by $V_i$ an arbitrary subspace of $V_5$ of dimension $i$.

All lines in $G(2,V_5)$ are of the   type:
$$\{ [V_2] \mid V_1\subset V_2\subset V_3\}.$$

Any $2$-plane  in    $G(2,V_5)$ is of one of the following     types:
\begin{itemize}
\item an $\alpha$-plane: $\{ [V_2] \mid  V_1\subset V_2\subset V_4\}$;
\item a $\beta$-plane: $\{ [V_2] \mid   V_2\subset V_3\} $.
\end{itemize}

Any  (possibly nonreduced or reducible) conic $c$ in $G(2,V_5)$ is of one of the following types:
\begin{itemize}
\item a {\em $\tau$-conic}: the $2$-plane $\langle c\rangle$ is not contained in $G(2,V_5)$, there is a unique hyperplane $V_4\subset V_5$ such that $c\subset G(2,V_4)$, the conic $c$ is reduced and, if it is smooth, the union of the corresponding lines in $\P(V_5)$ is a smooth quadric surface in $\P(V_4)$;
\item a {\em $\sigma$-conic}: the $2$-plane $\langle c\rangle$ is an $\alpha$-plane, there is a unique hyperplane $V_4\subset V_5$ such that $c\subset G(2,V_4)$,  and the union of the corresponding lines in $\P(V_5)$ is a quadric cone in $\P(V_4)$;
\item a {\em $\rho$-conic}: the $2$-plane $\langle c\rangle$ is a $\beta$-plane and the union of the corresponding lines in $\P(V_5)$ is this $2$-plane.
\end{itemize}

\subsection{The fourfold $W$}\label{fourw}   
Choose a   linear space $\P_7=\P(V_8)$ of codimension $2$ in $\P_9=\P(\wedge^2 V_5)$  whose dual pencil     $\P(V_8^\bot)\subset \P(\wedge^2 V_5^\vee)$ consists of skew forms on $V_5$, all of maximal rank $4$. This is equivalent to saying that the intersection
$$W=G(2,V_5)\cap \P_7 \subset \P(\wedge^2 V_5)
$$
is a smooth fourfold (of degree $5$). 

The map $ \P(V_8^\bot) \to \P ( V_5 )$ that sends a skew form to its kernel has image a smooth conic $c_U$ that spans a $2$-plane $\P (U_3)$ (\cite{pv}, Proposition 6.3), where   $U_3\subset  V_5$ is the unique common maximal isotropic subspace for all forms in   $ V_8^\bot $.
A normal form for matrices in the pencil $  \P(V_8^\bot)$ is given in \cite{pv}, Proposition 6.4; it shows that all $W$ are isomorphic under the action of $\PGL(V_5)$.

More precisely, if we choose $U_2$ such that
$V_5=U_2\oplus U_3$, we have  $U_3\isom \Sym^2U_2^\vee $. As   ${\mathfrak sl}_2$-modules, 
$$\wedge^2V_5\isom U_1\oplus U_2\oplus U_3\oplus U_4,$$
where $U_k$ is an irreducible ${\mathfrak sl}_2$-module of dimension $k$. In particular, 
$\wedge^2V_5$ contains a unique codimension-two submodule $V_8=U_1\oplus U_3\oplus U_4$, which can also be defined as the kernel of the natural map 
$$\wedge^2V_5\rightarrow U_2\otimes U_3\ra U_2^\vee,$$
where the rightmost map is the contraction $x\otimes q\mapsto q(x,\cdot )$, for $x$ in $U_2$ and 
$q$ a   symmetric bilinear  form on $U_2$. One can check that $\P (V_8)$ meets the Grassmannian 
$G(2,V_5)$ tranversely, so that the intersection is our smooth fourfold $W$.

\subsection{$2$-planes   in $W$}\label{2pla}   The $2$-plane
$\Pi=G(2,U_3) $ is the unique $\beta$-plane of $G(2,V_5)$ contained in $W$. 

An $\alpha$-plane contained in $W$ corresponds to a pair $V_1\subset V_4\subset V_5$ such that $\omega(v,w) =0$ for all $v\in V_1$, all $w\in V_4$, and all $\omega$ in the pencil $\P(V_8^\bot)$. It follows that $[V_1]$ must be in $c_U$ and $V_4$ is the common orthogonal of $v$ for all $\omega$ in $\P(V_8^\bot)$. There is a   line $L_U\subset \P(V_5^\vee)$ of such $2$-planes $\Pi_{V_4}$, corresponding to hyperplanes $V_4\subset V_5$ that contain $U_3$.

The intersection $\Pi\cap\Pi_{V_4}$ is a line which is tangent to the conic $c_U^\vee\subset\Pi$ dual to $c_U$, whereas two distinct $\Pi_{V_4}$ and $\Pi_{V'_4}$ meet at a point, which is on $\Pi$.

\subsection{The vector bundle $\cM$}\label{iden}
 For any hyperplane $ V_4\subset V_5$, the vector space 
 $M_{V_4}= \wedge^2 V_4 \cap V_8$
  has dimension $4$,\footnote{Otherwise, some form in the pencil $V_8^\bot$ would vanish on $V_4$ hence would have rank $\le 2$ (\cite{lo}, Lemma 3.1).} so this defines a rank-$4$ vector bundle $$\cM\to \P(V_5^\vee)$$
with fiber $M_{V_4}$ at $[V_4]$.
 The $3$-plane $\P(M_{V_4})$ contains the quadric surface\footnote{ This is indeed a surface since $\Pic(W)$ is generated by $\cO_W(1)$, hence all threefolds in $W$ have degree divisible by $5$.} 
\begin{equation}\label{qw}
Q_{W,V_4}=G(2,V_4)\cap \P(M_{V_4})\subset W
\end{equation}
which  is reducible if and only if it is of the form $\Pi\cup\Pi_{V_4}$; this happens if and only if $U_3\subset V_4$, \ie, $[V_4]\in L_U$.

\subsection{Automorphisms of $W$}\label{sautw}
Any automorphism of $W$ is induced by an automorphism of $\P(V_5)$ that maps the conic $c_U$ onto itself (\cite{pv}, Theorem 1.2). The automorphism group of $W$ has dimension $8$ and fits into an exact sequence  (\cite{pv}, Theorem 6.6)
\begin{equation}\label{autw}
1\to \C^4\ltimes \C^*\to \Aut(W) \to \Aut(c_U)\isom \PGL(2,\C)\to 1.
\end{equation}
For later use, we will describe this sequence more explicitely  with the help of a decomposition 
$V_5=U_2\oplus U_3$ as in \S\ref{fourw}. Let   $\Aut^+(W)\subset \GL(V_5)$ be the pull-back of $\Aut(W)\subset \PGL(V_5)$. 
As we saw in \S\ref{fourw}, any $\varphi\in\Aut^+(W)$ must preserve $U_3$, so it decomposes as 
$$\varphi=\left( \begin{matrix} \phi_2 & 0 \\ \phi_0 & \phi_3 \end{matrix}\right),$$
and the projection $\varphi\mapsto\varphi_2$ induces the map $\Aut(W) \to \Aut(c_U)$ in (\ref{autw}).
It also needs to preserve $V_8\subset \wedge^2V_5$, and this implies that
$\varphi_3$ must be a nonzero multiple of $\Sym^2{}^t\varphi_2 $, the map induced by $\varphi_2$
on $U_3=\Sym^2U_2^\vee$; this nonzero multiple gives the $\C^*$-factor in the sequence above. 

The   $\C^4$-factor corresponds to the case where $\varphi_2$ and $\varphi_3$ are   the identity, in which case $\phi_0\in \Hom(U_2,U_3)=\Hom(U_2,\Sym^2U_2^\vee)$ must be such that 
$$\forall u,v\in U_2  \quad \varphi_0(u)(v, \cdot)= \varphi_0(v)(u, \cdot) .$$ 
This is equivalent to the condition that $\varphi_0$ be completely symmetric, that is, belong  to the image 
of the natural map $\Sym^3U_2^\vee\to \Hom(U_2,\Sym^2U_2^\vee)$. In particular, the $\C^4$-factor in the sequence above is 
a copy of $\Sym^3U_2^\vee$ as an $\SL(2,\C)$-module.

Finally, the group $\Aut(W)$ acts on $W$ with four orbits (\cite{pv}, Proposition 6.8):
\begin{itemize}
\item $O_1=c_U^\vee\subset\Pi$,
\item  $O_2=\Pi\moins c_U^\vee$,
\item  $O_3=\bigcup_{V_4\in L_U}\Pi_{V_4}\moins\Pi$,
\item   $O_4=W\moins O_3$, 
\end{itemize}
where $\dim(O_j)=j$.

\subsection{Rationality of $W$}\label{birw}

Once we have chosen, as in \S\ref{fourw}, a decomposition $V_5=U_2\op U_3$, with $U_3=\Sym^2U_2^\vee$, we get
an induced action of $\GL(2,\C)$ on $W$. 

\begin{prop}
The smooth fourfold $W$ is a compactification of the affine homogeneous space 
$\GL(2,\C)/\mathfrak{S}_3$. 
\end{prop}

\begin{proof}
If $x$ is a general point in $W\subset G(2,V_5)$, the corresponding subspace in $V_5$ maps isomorphically to $U_2$, 
and is therefore   the graph of a map $g_x: U_2\to U_3$. The map $y\mapsto
g_x(y)(y)$ is a general cubic form on $U_2$, and the stabilizer of $x$ in $\GL(2,\C)$ must preserve the 
zeroes of this cubic form, hence be equal to the symmetric group $\mathfrak{S}_3$.
\end{proof}

As a consequence, $W$ is rational. For the proof of Logachev's reconstruction theorem \ref{rec}, we will need a 
precise description of an explicit birational isomorphism with $\P^4$. Recall  that   $W$ sits in $\P(V_8)=\P(U_1\op U_3\op U_4)$.

\begin{prop}\label{kappa}
The projection from $\P (U_3)$ defines a birational isomorphism    $\kappa :W\dra\P(U_1\op U_4) $.
The inverse $\kappa^{-1}$ is defined by the linear system of quadrics containing a   rational normal
cubic curve   $ \Gamma^0_3\subset \P(U_4)\subset \P(U_1\op U_4)$.
\end{prop}

\begin{proof}
A point in $G(2,V_5)$ corresponds to a tensor of the form 
$$(u+P)\wedge (v+Q)=u\wedge v+u\ot Q-v\ot P+P\wedge Q,$$
 with $u$ and $v$ in $ U_2$, and $P$ and $Q$ in $ U_3\isom \Sym^2U_2^\vee$. It belongs to  $W$ if and only if $Q(u,\cdot)=P(v,\cdot)$ as linear forms on $U_2$. 
We must prove that the projection sending this tensor to $u\wedge v+u\ot Q-v\ot P$ is generically
injective. This follows   from the fact there there is a natural map
$$\Sym^2U_4\hookrightarrow \Sym^2(U_2\ot U_3)\ra\wedge^2U_2\ot \wedge^2U_3\simeq U_3,$$
sending   $(u\ot Q)(v\ot P)$ to $(u\wedge v)(P\wedge Q)$,   from which we can get the  component $P\wedge Q$ back.

Moreover, this shows that the inverse map is defined by the space of quadrics orthogonal
to the kernel of this morphism in $\Sym^2(U_1\op U_4)^\vee$, and these are precisely the
quadrics containing  the rational normal
cubic curve $\Gamma^0_3$ image of $\P(U_2^\vee)$ in $ \P(U_4)=\P(\Sym^3U_2^\vee)$.
\end{proof}

The rational map $\kappa$ is not defined along the $2$-plane $\Pi$, whose total image is
  the hyperplane 
$\P (U_4)$. More precisely, through any point $x$ in $\Pi$, there are two (possibly equal) lines tangent to $c_U^\vee$, which meet this conic at (possibly equal)  points $x_1$ and $x_2$, and the total image of $x$ by $\kappa$ is the line $\ell_x$ bisecant (or tangent) to  $\Gamma^0_3$ through $x_1$ and $x_2$.


If a hyperplane $V_4\subset V_5$ does not contain $U_3$, the quadric surface $ Q_{W,V_4}$  (see (\ref{qw}))
 meets $\Pi $ at the unique
point   $[U_3\cap V_4]$, so the restriction of $\kappa$ to $Q_{W,V_4}$ sends it to a $2$-plane in $\P(U_1\op U_4)$.
If $V_4$ does contain $U_3$, that is,  if $V_4=V_1^{\perp}$ for some line $V_1\subset U_2$ defining
a point in $ c_U$,    the linear span of $Q_{W,V_4}$ is mapped to the tangent line to
the corresponding point of $\Gamma_3^0$.

\subsection{Cohomology calculations}\label{coca}
We gather here various results on the cohomology groups of some sheaves of twisted differentials on $W$.

\begin{prop}\label{cogr}We have
\begin{eqnarray}
H^3(W,\Omega^1_W(-3))&=&0\label{-3},\\
H^q(W,\Omega^1_W(-2))&\isom&\delta_{q,3}L\label{-2},\\
H^q(W,\Omega^2_W(-1))& \isom&\delta_{q,3}V_5/U_3.\label{-1}
\end{eqnarray}
\end{prop}

\begin{proof}
Set $G=G(2,V_5)$. By the Bott-Borel-Weil theorem, the sheaves $\Omega^1_G(-r)$ for $1\le r\le 4$, and $\Omega^2_G(-r)$ for $1\le r\le 2$,  are acyclic, whereas
$$H^q(G, \Omega^2_G(-3))=\delta_{q,5}V_5.
$$
Using the Koszul resolution
$$0\to\wedge^2L\otimes\cO_G(-2)\to L\otimes\cO_G(-1)\to\cO_G\to\cO_W\to 0,$$
  we obtain that $\Omega^1_G(-2)\vert_W$ is acyclic, and
    \begin{equation}\label{fz11}
    H^q(W,\Omega^2_G(-1)\vert_W)\simeq H^{q+2}(G,\Omega^2_G(-3))\otimes \wedge^2L \isom \delta_{q,3}\wedge^4V_5^\vee.
    \end{equation}
Using the conormal sequence
  \begin{equation}\label{conor}
0\to L\ot\cO_W(-1)\to \Omega^1_G\vert_W\to \Omega^1_W\to 0, 
\end{equation}
we obtain 
$$H^q(W,\Omega^1_W(-2))\simeq H^{q+1}(W,\cO_W(-3))\otimes L  $$
and (\ref{-2}) follows. 

The vanishing of $H^3(W,\Omega^1_W(-3))$ is a little less straightforward. 
From  (\ref{conor}), we deduce an exact sequence
$$0\to H^3(W,\Omega^1_W(-3))\to L\ot H^4(W,\cO_W(-4))\to 
H^4(W,\Omega^1_G(-3)\vert_W).$$
The vector space $H^4(W,\cO_W(-4))$ is Serre-dual to $H^0(W,\cO_W(1))=
\wedge^2V_5^\vee/L$. 
The acyclicity of $\Omega^1_G(-r)$ for $r\in\{3,4\}$ yields 
$$H^4(W,\Omega^1_G(-3)\vert_W)\simeq H^6(G,\Omega^1_G(-5)),$$ 
which is Serre-dual to $H^0(G,T_G)\simeq {\mathfrak sl}(V_5)$.
By duality, we are therefore reduced to verifying that the natural map
$$ {\mathfrak sl}(V_5^\vee)\to \Hom(L,\wedge^2V_5^\vee/L)$$
is surjective. This is checked by an explicit computation, and   proves (\ref{-3}).

 The exact sequence (\ref{conor}) also induces exact sequences
\begin{eqnarray}
 &0\to \cK\to\Omega^2_G\vert_W\to\Omega^2_W\to 0, \label{fz1}\\
 &0\to \wedge^2L\otimes\cO_W(-2)\to \cK\to 
L\otimes\Omega^1_W(-1)\to 0.\label{fz2}
\end{eqnarray}
From   (\ref{-2}) and (\ref{fz2}), it follows that the cohomology groups of $\cK(-1)$ are  
$0$ in degrees $0$, $1$, and $2$, and there is an exact sequence
$$0\to H^3(W,\cK(-1))\to L\otimes L\to \wedge^2L\to H^4(W,\cK(-1))\to 0.$$
Therefore $H^q(W,\cK(-1))\isom \delta_{q,3} \Sym^2L$. With (\ref{fz1}) and  (\ref{fz11}), we obtain  that $\Omega^2_W(-1)$ has zero cohomology in degree 
$0 $, $1$, and $4$, and that there is an exact sequence 
$$0\to H^2(W,\Omega^2_W(-1))\to \Sym^2L\stackrel{\psi}{\to} \wedge^4V_5^\vee\to 
H^3(W,\Omega^2_W(-1))\to 0.$$
The   map $\psi  $ is the obtained from   the
inclusion $L\hookrightarrow\wedge^2V_5^\vee$ and the natural map 
$\Sym^2(\wedge^2V_5^\vee)\to\wedge^4V_5^\vee$ defined by the wedge product. 
One   checks that the noninjectivity of $\psi$   would imply
the existence of a rank-$2$ form in $L$, which is not the case. Also,
if we identify $\wedge^4V_5^\vee$ with $V_5$, the image of $\psi$ is 
the subspace we denoted by  $U_3$. We have therefore proved (\ref{-1}).
\end{proof}

\section{The Fano threefold $X$}\label{fanox}

 \subsection{The two types of Fano threefolds of type $\cX_{10}$}\label{gush}
 
  Let $X$ be a smooth projective  complex threefold with Picard group $ \Z[K_X]$ and $(-K_X)^3=10$. The linear system $|-K_X|$ is   very ample and embeds $X$ in $\P^7$ as a smooth subvariety of degree $10$. Gushel and Mukai proved that it is of one of the following types described below (\cite{gu}, \S 4; \cite{ip}, Corollary 4.1.13 and Theorem 5.1.1).
 
  Let $V_5$ be a $5$-dimensional   vector space. Consider the Grassmannian $G(2,V_5) \subset \P(\wedge^2 V_5)$ in its Pl\"ucker embedding and denote by $\P_m$ a linear subspace of $\P(\wedge^2 V_5)$ of dimension $m$. Then: 
 \begin{itemize}
 \item either 
 \begin{equation}\label{defX}
X =G(2,V_5)\cap \P_7 \cap \Omega =W\cap \Omega \subset \P_7,
\end{equation}
 where  $\Omega$ is a quadric;
 \item or $X$ is the intersection in $\P_7$ of a cone over  $ G(2,V_5)\cap \P_6$ with a quadric.
  \end{itemize}
The second case (``Gushel threefolds'') is a degeneration of the first case.\footnote{Intersect, in $\P^{10}$, a cone over $ G(2,V_5)$  with a quadric and a $\P_7$; if the $\P_7$ does not pass through the   vertex of the cone, we are in the first case; if it does, we are in the second case. Gushel threefolds are missing in   Iskovskikh's 1977 classification  of Fano threefolds with Picard number $ 1$ (see   tables    in \cite{Is11}, p. 505).} Gushel threefolds were studied in \cite{gus}. 

 We will only consider threefolds of the first type, which we denote by $\cX_{10}^0$. They depend on $22$ parameters (see \S\ref{modu}), whereas Gushel threefolds depend on only $19$ parameters.

\subsection{The moduli space $\cM_X(2;1,4)$} 
Let us recall how the embedding of  $X$ into $ G(2,V_5)$ was obtained by Gushel and Mukai. There exists on $X$ a smooth elliptic quartic curve $\Gamma^1_4$ whose linear span is $3$-dimensional (\cite{ip}, Lemma 5.1.2). 
 Serre's construction yields a rank-$2$ vector bundle $\cE$ on $X$  with a section 
vanishing exactly on $\Gamma^1_4$  fitting into an extension
$$0\to\cO_X\to \cE\to\cI_{\Gamma^1_4}(1)\to 0.$$
We have $c_1(\cE)=[\cO_X(1)]$ and $c_2(\cE)=[\Gamma^1_4]$. Moreover, $h^0(X,\cE)=5$, 
$H^i(X,\cE)=0$ for $i>0$, and $H^i(X,\cE^\vee)=0$ for $i\ge 0$. In particular 
$H^0(X,\cE^\vee)= H^0(X,\cE(-1))=0$, so $\cE$ is stable since $\Pic(X)=\Z[\cO_X(1)]$.
By \cite{ip}, Lemma 5.1.3, $\cE$ is globally generated and defines a morphism 
of $X$ into $G(2,H^0(X,\cE)^\vee)=G(2,V_5)$ such that $\cE$ is the restriction to $X$
of the dual tautological rank-$2$ bundle $\cS_{2,V_5}^\vee$ on $G(2,V_5)$.\footnote{This morphism is an embedding, except for Gushel threefolds (\S\ref{gush}), for which it induces a double cover of a del Pezzo threefold $Y_5\subset \P^6$ (\cite{ip}, \S5.1).}

\begin{prop}Let $X$ be any smooth Fano threefold   of type $\cX_{10}$.
The vector bundle $\cE$ is the unique stable rank-$2$ vector bundle on $X$
with Chern numbers $c_1(\cE)=1$ and $c_2(\cE)=4$.
\end{prop}

\begin{proof} We follow \cite{muk}.
Let $\cF$ be a vector bundle on $X$ with the same properties as $\cE$. 
Let $\cH=\cH om(\cE,\cF)$. A general hyperplane section $S$ of $X$ is a $K3$ surface, and again $\Pic(S)=\Z[\cO_S(1)]$.  

The restriction $\cH\vert_S$  has Chern numbers $c_1(\cH\vert_S)=0$ and 
$c_2(\cH\vert_S)=4c_2(\cE\vert_S)-c_1(\cE\vert_S)^2=6$. So by the  Riemann-Roch theorem,
$$\chi(S,\cH\vert_S)=4\chi(\cO_S)+\frac{1}{2}(c_1(\cH\vert_S)^2-2c_2(\cH\vert_S))=2.$$ 
In particular $h^0(S,\cH om(\cE\vert_S,\cF\vert_S))$
or $h^0(S,\cH om(\cF\vert_S,\cE\vert_S))$ must be  nonzero. 
Since $\cE\vert_S$ and $\cF\vert_S$ are stable bundles 
with the same slope, this implies that they are isomorphic. 

We want to prove that any isomorphism between $\cE\vert_S$ and $\cF\vert_S$ extends to $X$, 
by showing   $H^1(X,\cH(-1))=0$. We   prove by descending induction on $k$ that $H^1(X,\cH(-k))$ vanishes
for all $k>0$. Because of the exact sequence $0\to\cO_X(-1)\to\cO_X\to\cO_S\to 0$, 
we just need to check that $H^1(S,\cH\vert_S(-k))$ vanishes for all $k>0$. 

Now, since $\cE$ and $\cF$ are both isomorphic on $S$ to the restriction of 
 $\cS_{2,V_5}^\vee$, we have $\cH\vert_S=\cE nd(\cS_{2,V_5}^\vee)\vert_S$. 
Since $S$ is a codimension-$4$ subscheme of $G(2,V_5)$ defined by the vanishing of a section of the 
vector bundle $\cN=\cO_{G(2,V_5)}(1)^3\oplus\cO_{G(2,V_5)}(2)$, we just need to check
that  $H^{i+1}(G(2,V_5), \cE nd(\cS_{2,V_5}^\vee)\otimes\wedge^i\cN^\vee)$ vanishes for $0\le i\le 4$. This is an 
easy consequence of the Bott-Borel-Weil theorem. 
 \end{proof}

\begin{coro}
Any isomorphism between two   Fano threefolds of type $\cX_{10}^0$  is induced by an automorphism of $\P(V_5)$ 
preserving $W$. 
\end{coro}

\subsection{Automorphisms} 

\begin{theo}\label{autofini}
For any smooth Fano threefold $X $ of type $\cX^0_{10}$, we have
$$H^i(X,T_X)=0 \qquad \hbox{\it for}\quad i\ne 1.$$
In particular, the group of automorphisms of $X$ is finite.
\end{theo}

\begin{proof}
For $i\ge 2$, this follows from the Kodaira-Akizuki-Nakano (KAN) vanishing 
theorem since $T_X\simeq\Omega_X^2(1)$. The conormal exact sequence
\begin{equation}\label{conorWX}
0\to\cO_X(-2)\to\Omega^1_W\vert_X\to\Omega^1_X\to 0 
\end{equation}
  induces a resolution  of $\Omega^2_X$, hence of 
$T_X\simeq\Omega^2_X(1)$:
\begin{equation}\label{resol}
0\to\Omega^1_X(-1)\to\Omega^2_W(1)\vert_X\to T_X\to 0.
\end{equation}
By KAN vanishing 
  again, we deduce  $H^0(X,\Omega^2_W(1)\vert_X)\simeq H^0(X,T_X)$.  Since $X$ is a quadratic section of $W$, there is an exact sequence
\begin{eqnarray*}
 0\rightarrow \Omega_{W}^2(-1)\rightarrow \Omega_{W}^2(1)
\rightarrow  \Omega_W^2(1)\vert_X\rightarrow 0.
\end{eqnarray*}
Using   KAN vanishing 
  again (on $W$), we obtain an isomorphism\break $H^0(W,\Omega^2_W(1) )\simeq H^0(X,\Omega^2_W(1)\vert_X) $. Since $ \Omega^2_W(1)$ is Serre-dual to $\Omega^2_W( -1)$, the theorem follows from (\ref{-1}).
  \end{proof}

\begin{theo}\label{noauto}
A general   Fano threefold  of type $\cX_{10} $ has no nontrivial  automorphisms.  
\end{theo}

\begin{proof}
Assume  $\varphi\in\GL(V_5)$ induces a nontrivial automorphism of $W$. We want to prove that the 
space of polynomials $Q\in H^0(\P_7,\cO_{\P_7}(2))=\Sym^2V_8^\vee$ such that $\varphi$ preserves $X= \P_7\cap \{Q=0\}$ has 
codimension bigger than $7=\dim (\Aut(W))$. 

This condition is equivalent to the 
fact that there exists a nonzero scalar $z$ such that $\varphi^*Q=zQ+P$ for 
some $P\in H^0(\P_7,\cI_W(2))$. So the dimension we are interested in is controlled by 
the dimension of the eigenspaces of $\varphi^*$. To estimate 
the dimensions of these eigenspaces, we may by semicontinuity 
assume that the image of $\varphi$ in $\PGL(2,\C)$ (see (\ref{autw})) is the identity.  The 
eigenvalues of $\varphi$ on $V_5$ are then $1$, with multiplicity $2$, and some $\zeta^{-1}$ 
with multiplicity $3$. The eigenvalues of $\varphi^*$ on $\Sym^2V_8^\vee$ 
are thus $1$, $\zeta$, $\zeta^2$, $\zeta^3$, $\zeta^4$, with respective multiplicities $1$, $4$, $13$, $12$, and $6$. 
These eigenvalues can coincide for certain values of $\zeta$, but it is easy to check 
that for $\zeta\ne 1$, no eigenspace can have dimension  $>20$. So the
maximal number of parameters for $Q$ is 
$$20+h^0(\P_7,\cI_W(2))=25 <36-7.$$ 

Suppose now   $\zeta=1$.  As we saw in \S\ref{fourw}, this means that in a decomposition $V_5=U_2\op U_3$, we have
$$\varphi=\left( \begin{matrix} I_2 & 0 \\ \phi_0 & I_3 \end{matrix}\right),$$
with $\varphi_0$ in $ \Sym^3U_2^\vee\subset \Hom(U_2,\Sym^2U_2^\vee)$ completely symmetric (but nonzero).  
By semicontinuity again, we may suppose that $\varphi_0$ is of the form $\ell^3$ for some linear form 
$\ell$ on $U_2$, so that 
$$\forall u,x,y\in U_2\quad \varphi_0(u)(x,y)=\ell(u)\ell(x)\ell(y)   .$$
Thus $\varphi_0$ has rank $3$, and a straightforward computation shows that the 
endomorphism $\Phi$ of $\Sym^2V_8^\vee$ induced by $\varphi$ is such that $\Id-\Phi$ has rank $18$. 
Therefore we get at least $18-5>7$ conditions on $Q$.  This concludes the proof. \end{proof}

\begin{rema}\label{codim}
The proof shows that  Fano threefolds of type $\cX_{10}^0$ with nontrivial automorphisms form a family of codimension $\ge 4$.
\end{rema}

\subsection{Moduli}\label{modu}

The most natural way to define an moduli ``space'' for Fano varieties of type $\cX_{10}^0$ would be as the quotient of an open subset of $ |\cO_W(2))|$ by the action of $\Aut(W)$. However, the latter group being nonreductive, the question of whether this quotient is a scheme is difficult to settle. On the other hand, it is not difficult to construct this same ``space'' as  the quotient of a quasi-projective variety  by the action of the (reductive) group $\SL(V_5)$, but again, it is not clear whether this subset corresponds to (semi)stable points for some polarization, or even that the action  is proper. We hope to come back to this interesting question in the future, but for the time being, we will just consider the algebraic   {\em stack}
 $\cN_{10}$ of Fano threefolds of  type $\cX_{10}^0$. It is an algebraic, smooth, irreducible  stack of  dimension   
$$\dim |\cO_W(2))|-\dim(\Aut(W))=30-8
= 22.
$$
This will be sufficient for our purposes. Locally around a point corresponding to a threefold $X$, this stack is given by   the local Kuranishi space of $X$,   a smooth 22-dimensional variety, acted on by the finite group $\Aut(X)$.
We   denote by  $\wp:\cN_{10}\to\cA_{10}$ the period map.

\section{The varieties of conics contained in $X$}\label{sec5} 

Most of the material in this section is due to Logachev (\cite{lo}).

\subsection{The surfaces $ F_g(X)$ and $F(X)$}\label{surF}
Let $X$ be a Fano threefold of type $\cX_{10}^0$.
Let $ F_g(X)$   be the variety of  (possibly nonreduced or reducible) conics contained in $X$.  It follows from deformation theory and \cite{ip}, Proposition 4.2.5.(iii),  that $F_g(X)$ has  dimension $2$ everywhere.

As seen in \S\ref{pla}, any conic $c$ in $G(2,V_5)$ is contained in some $G(2,V_4)$, where the hyperplane $V_4\subset V_5$ is uniquely determined by $c$ unless $\langle c\rangle$ is a $\beta$-plane.
It is therefore natural to introduce the incidence variety
$$F(X)=\{(c,[V_4]) \in F_g(X)\times \P(V_5^\vee) \mid  c\subset G(2,V_4)\}.
$$
The first projection $p_1:F(X)\to F_g(X)$ is an isomorphism except over the point $[c_X]$  that corresponds to the $\rho$-conic $\Pi\cap \Omega$, where the fiber $L_\rho$ is isomorphic to the line $L_U$ via the second projection $p_2:F(X)\to \P(V_5^\vee)$. The   conic $c_X$ is the only $\rho$-conic on $X$.  

For any $[V_4]$ in $L_U$, the intersection $\Pi_{V_4}\cap \Omega$ is   a $\sigma$-conic in $X$.   These conics describe a curve $L_\sigma\subset F(X)$ isomorphic to $L_U$ via $p_2$.  These are the only   $\sigma$-conics on $X$.  

If $c$ is a conic contained in $X$, its span $\langle c\rangle$ is a $2$-plane such that
$\langle c\rangle\cap X=c$.\footnote{This is because $X$ contains no $2$-planes (footnote \ref{lef}) and is an intersection of quadrics.}

If $(c,V_4)$ is a  $\sigma$-conic, the intersection $\Pi\cap\Pi_{V_4}$ is a line  in $\Pi$ tangent to the conic $c_U^\vee$. In general, it meets $c_X$ in the  two points of $c_X\cap c$. Through any of these two points passes one other tangent to $c_U^\vee$, hence $c$ meets exactly two other $\sigma$-conics, at points of $c_X$. In particular, two general $\sigma$-conics are disjoint. Also,  there are two   $\sigma$-conics    through a general point of $c_X$.

Any point of a $\sigma$-conic $(c,V_4)$ corresponds to a line in $\P(V_4)$, which must meet $\P(U_3)$. Therefore, the union of all $\sigma$-conics in $X$ is contained in its   section with the hyperplane
$\{ V_2\subset V_5\mid V_2\cap U_3\ne\{0\}\} $. This section being irreducible, they are equal.

\subsection{The involution $\iota $ on $F(X)$}\label{invo}

Let $V_4\subset V_5$ be a hyperplane.
In the $3$-plane  $\P(M_{V_4})= \P(\wedge^2 V_4) \cap \P_7$ introduced   in \S\ref{iden}, let us define the quadric 
$$Q_{\Omega,V_4} =  \Omega\cap \P(M_{V_4})$$
and set (see (\ref{qw}))
\begin{equation}\label{gamma}
\Gamma^1_{4,V_4}=X\cap \P(M_{V_4})=Q_{W,V_4}\cap Q_{\Omega,V_4},
\end{equation}
which is      
 a genus-$1$ degree-$4$ $1$-cycle.\footnote{\label{lef}Since $\Pic(X)$ is generated by $\cO_X(1)$, all surfaces in $X$ have degree divisible by $10$, hence $ \Gamma^1_{4,V_4}$ is a curve.} 
We have
$$(c,V_4)\in F(X) 
 \Longleftrightarrow   c\subset \Gamma^1_{4,V_4}.
$$
{\em Assume now  that $X$ is general.} Easy parameters counts (see \cite{lo}, Lemma 3.7) then show that the residual curve is another conic $\iota(c)\subset X$ that meets $c$ in two points. This defines a fixed-point-free involution $\iota$ on $F(X)$ and the quotient $F_\iota(X)=F(X)/\iota$ maps injectively to $\P(V_5^\vee)$ by the projection $p_2$. 

 For any $[V_4]$ in $L_U$, the quadric  $ Q_{W,V_4}$ is already reducible (see \S\ref{iden}), hence $L_U\subset p_2(F(X))$. Moreover, $p_2^{-1}([V_4])$ has exactly two points, which correspond to the conics $\Pi\cap \Omega$ and $\Pi_{V_4}\cap \Omega$. We have $p_2^{-1}(L_U)=L_\rho\sqcup L_\sigma$ and $\iota (L_\rho)=L_\sigma$.

Here is a     characterization of the involution $\iota$.

\begin{lemm}\label{l51}
Let $c$ and $c'$ be two conics on a general $X$ of type $\cX_{10}$, with no common component. If  
$\dim(\langle c,c'\rangle)=3$ and $\deg(c\cdot c')=2$, we have $c'=\iota(c)$.
\end{lemm}

\begin{proof}
If $\langle c\rangle$ is  contained in $W$, the line $\langle c\rangle\cap  \langle c'\rangle$ and the conic $c'$ are in $W\cap  \langle c'\rangle$. The line is not contained in $c'$ (because  $X\cap  \langle c \rangle=c$) hence   $W$, which is an intersection of quadrics, contains $\langle c'\rangle$. One of the conics is then $c_X$ and the other is a $\sigma$-conic so we are done.

Since $W$ is an intersection of quadrics, if it contains neither planes $\langle c\rangle$ and $\langle c'\rangle$,  it does not contain the line  $\langle c\rangle\cap  \langle c'\rangle$. Any point $z$ on the line $\langle c\rangle\cap  \langle c'\rangle$ but not on $c\cup c'$ is not on $G(2,V_5)$.
We may write $x=v_1\wedge v_2$, $y=v_3\wedge v_4$, and $z=v_1\wedge v_2+v_3\wedge v_4$. The vectors $v_1,\dots,v_4$ span a hyperplane  $V_4\subset V_5$ and every  bisecant line to $G(2,V_5)$ passing through $z$ is contained in $\P(\wedge^2 V_4)$.  It follows that $\P(M_{V_4})$ contains $c$ and $c'$, hence $c'=\iota(c)$.\end{proof}

\begin{rema}\label{any}
One checks, by a case-by-case analysis, that the involution $\iota$ can still be defined on $F(X)$ for {\em any} smooth $X$. It can however have   fixed points. More precisely,  for a smooth conic $c\subset X$, the following conditions are equivalent:
\begin{itemize}
\item[(i)] $\iota(c)=c$;
\item[(ii)] $c$ is a $\tau$-conic with normal bundle   $N_{c/X}\isom\cO(2)\oplus\cO(-2)$;
\item[(iii)] there is a $\P^3\subset\P_7$ such that $X\cap\P^3=2c$ (this generalizes Lemma \ref{l51}).
\end{itemize}
Here is a quick proof. If   (i) holds, we have $X\cap \P(M_{V_4})=2c$, hence (iii) holds. If (iii) holds, the conic $c\subset \P^3$ is
  the intersection of a double plane and a quadric $Q$. With the notation of \S\ref{ec}, the normal direction to $c$ in $Q$ at each point of $c$ corresponds to a   curve  in the exceptional divisor $E$ which is contracted by the projection from $c$, hence $N_{c/X}$ must be $\cO(2)\oplus\cO(-2)$, and (ii) holds.
Finally, if $(c,V_4)$ is a $\tau$-conic, $N_{c/W}\isom \cO(2)\oplus\cO(1)\oplus\cO(1)$. In particular, the $\cO(2)$-factor is uniquely determined, and it is the image of $N_{c/Q_{W,V_4}}$. If (ii) holds, the $\cO(2)$-factor of $N_{c/X}$ must map to $N_{c/Q_{W,V_4}}$, hence the induced morphism $N_{c/Q_{W,V_4}} \to N_{X/W}$ must be zero. This means that the quadric   $\Omega$ vanishes to order $2$ on $c$, hence (i) holds.

These conics correspond to singular points of the surface $F(X)$.
\end{rema}

\subsection{Logachev's tangent bundle theorem}\label{stbt} {\em Assume again  that $X$ is general.} 
 Logachev shows that $F(X) $ and $F_g(X)$ are smooth connected surfaces (\cite{lo}, Corollary 4.2). The map $p_1:F(X)\to F_g(X)$ is the contraction of the exceptional curve $L_\rho$ to $[c_X]$. The curve $L_\sigma=\iota(L_\rho)$ is therefore also exceptional. Let $F_g(X)\to F_m(X)$ be its contraction and let $r:F(X)\to F_m(X)$ be the composition.

   Let 
 $\beta:F(X)\to G(2,V_8)$ 
be the map that sends a conic $c$ to the projective line  $\langle c\cap \iota( c)\rangle$, and let $ \cS_{2,V_8} $ be the tautological bundle on $G(2,V_8)$.
 On the open  set $F (X)^0=F(X)\moins (L_\rho\sqcup L_\sigma)$, there is an isomorphism  (\cite{lo}, Theorem 4.14)
 \begin{equation*}\label{tbt}
  T_{F (X)^0}\isomlra  (\beta^*\cS_{2,V_8})\vert_{F(X)^0} . 
  \end{equation*}
   This  ``tangent bundle theorem''  has the following geometric interpretation (\cite{lo}, Theorem 7.2):   in the diagram
  $$
\xymatrix
{\P( T_{F_m(X)})\ar@{-->}[r]^-\psi\ar[d]^\pi& \P (H^0(F_m(X), \Omega_{F_m(X)})^\vee)\isom \P^9\ar@{-->}[d]^{\pi'}\\
F_m(X)&\P_7 }
$$
where $\psi$ is the {\em cotangent map} and $\pi' $ is the projection from the line $\psi(\pi^{-1}([c_X]))$, the composition $\pi'\circ \psi$ sends a general fiber $\pi^{-1}([c])=\P (T_{F_m(X),[c]})$ to the line $\langle  c \cap   \iota(c)\rangle$ in $\P_7$.

  \begin{coro}\label{min}
 The only rational curves   in $F(X)$ are $L_\rho$ and $L_\sigma$. In particular, the surface $ F_m(X) $ is minimal.
\end{coro}

 Since $\cS_{2,V_8}^\vee$ is generated by global sections on $G(2,V_8)$,  so is  $\Omega_{F_m(X) }$, except possibly at the points $[c_X]$ and $\iota([c_X])$. We will show later (Corollary \ref{gene}) that it is in fact globally generated everywhere.
 
 \begin{rema}
 For {\em any} $X$ of type $\cX_{10}^0$, the scheme $F(X)$ is still an irreducible surface (Corollary \ref{cor83}) which may have singular points  (Remark \ref{any}).
 
 Note that  the Fano surface   of a  Gushel threefold  (see \S\ref{gush}) has two components  (\cite{gus}, Proposition (2.1.2)).  
 \end{rema}

 \section{Elementary transformations}

\subsection{Elementary transformation along a conic}\label{ec} 
Let $c$ be a {\em smooth}  conic contained in a Fano threefold $X$ of type $\cX_{10}^0$ and let $\pi_c:\P_7\dra \P^4$ be the projection from the $2$-plane $\langle c\rangle$.
If $\eps:\widetilde X\to X$ is the blow-up of $c$, with exceptional divisor $E$, the composition 
$\pi_c\circ \eps : \widetilde X\to \P^4$ is the {\em morphism} defined by  
the linear system $|-K_{\widetilde X}|=|-\eps^*K_X-E|$. The only curves contracted by this morphism are 
(\cite{ip}, Proposition 4.4.1.(ii)):
\begin{itemize} 
\item the strict transforms of the lines in $X$  that meet $c$;
\item the strict transforms  of   conics $c'\ne c$ such that  $\dim(\langle c,c'\rangle)=3$ and $\deg(c\cdot c')=2$;
\item  when the normal bundle to $c$ in $X$ is $\cO(2)\oplus\cO(-2)$,  the exceptional section of the ruled surface $E\isom \F_4$.
\end{itemize} 

It is easy to check that only finitely many lines meet $c$ when {\em   either $c$ is general in $F(X)$}  (\cite{ip}, Lemma 4.2.6), or {\em$X$ itself is  general and $c$ is {\em any} smooth conic.}

{\em Assume from now on that $X$  is  general and $c\ne c_X$.} It follows from Lemma \ref{l51} that the only conic contracted by $\pi_c$ is $\iota(c)$, hence the morphism $\phi_{|-K_{\widetilde X}|} $ contracts only finitely many curves.\footnote{When $c$ is a $\tau$-conic, one can show that $\pi_c$  is injective on $X\moins c$, whereas when $c$ is a $\sigma$-conic, $\pi_c$ maps every other $\sigma$-conic $2$-to-$1$ onto a line.} 
 Let  
$$\phi_{|-K_{\widetilde X}|}:{\widetilde X}\stackrel{\phi}{\lra} \bar X\lra \P^4$$
be its Stein factorization. The variety $\bar X $ has terminal hypersurface singularities, $-K_{\bar X}$ is ample, and $\phi^*K_{\bar X}=K_{\widetilde X}$.  The divisor $-E $ is $\phi$-antiample. In this situation, there exists a {\em $(-E)$-flop} (\cite{ip}, Theorem 1.4.15)
 $$\chi:\widetilde X\stackrel{\phi}{\lra} \bar X \stackrel{\phi'}{\longleftarrow} \widetilde X'$$
which is an isomorphism in codimension $1$. The projective threefold ${\widetilde X}'$ is smooth and, if $\bar H $ is a hyperplane section of $\bar X$, we have $-K_{\widetilde X'}={\phi'}^*\bar H$  and $\chi(-E)$ is $\phi'$-ample. 

We have $\rho(\widetilde X')=2$. Since the extremal ray generated by the class of curves contracted by $\phi'$   has $K_{\widetilde X'}$-degree $0$ and $K_{\widetilde X'}$ is not nef, the other extremal ray is $K_{\widetilde X'}$-negative and defines a contraction $\eps':\widetilde X'\to X'$. We have (\cite{ip}, Proposition 4.4.11.(ii)):
\begin{itemize} 
\item $X'$ is again a smooth Fano threefold of degree $10$ in $\P^7$;
\item $\eps'$ is the blow-up of a smooth conic $c'$ in $X'$, with exceptional divisor   $E'\equiv - 2K_{\widetilde X'}- \chi(E) $.
\end{itemize} 
There is a commutative diagram
\begin{equation}\label{dia}
\xymatrix
{\widetilde X\ar[dd]_{\eps}\ar[dr]_\phi\ar@{-->}[rr]^\chi&&\widetilde X'\ar[dd]^{\eps' }\ar[dl]^{\phi'}\\
&\bar X\\
X\ar@{-->}[rr]^{\psi_c}\ar@{-->}[ur]^{\pi_c}&&X'\ar@{-->}[ul]_{\pi_{c'}}
}\end{equation}
 If $H'$ is a hyperplane section of $X'$, we have
$$\chi^*{\eps'}^*H'\equiv \chi^*(-K_{\widetilde X'}+E')\equiv \chi^*(-3K_{\widetilde X'}-\chi(E)) 
\equiv -3\eps^*K_X-4E 
$$
hence $\psi_c$ is associated with a  linear subsystem of  $|\cI_c^4(3)|$.\footnote{\label{note11}If $\ell$ is the strict transform in $\widetilde X$ of a line in $X$ that meets $c$, we have $(3\eps^*H-4E)\cdot\ell =-1$ hence $\ell$ is in the base locus of $|\chi^*{\eps'}^*H'|$. Similarly, the strict transform in $\widetilde X$ of the conic $\iota(c)$ has intersection $-2$ with  $3\eps^*H-4E$, hence is also   in the base locus of $|\chi^*{\eps'}^*H'|$. Calculations on the blow-up of $\iota(c)$ in $\widetilde X$ show that the base ideal is in fact contained in $\cI_{\iota(c)}^2$. So, for  $c$  general, the rational map $\psi_c$ is associated with a  linear subsystem of  $|\cI_{\ell_1}\otimes\cdots\otimes \cI_{\ell_{10}}\otimes \cI_{\iota(c)}^2 \otimes\cI_c^4(3)|$, where $\ell_1,\dots ,\ell_{10}$ are the ten lines in $X$ that meet $c$.
}

Note that ${\eps'}^*H'-E'\equiv  -K_{\widetilde X'} \equiv {\phi'}^* \bar H $, so the picture is symmetric: the elementary  transformation of $X'$ along the conic $c'$ is $\psi_c^{-1}:X'\dra X$ (by construction, $\phi'$ automatically contracts only finitely many curves).

\begin{prop}\label{prop62}
Let   $X$ be a general Fano threefold of type $\cX_{10}$ and let $c$ be a smooth $\tau$-conic   on $X$.
There is a   birational isomorphism
$$\phi_c: F_g(X) \dra F_g(X')$$ 
which commutes with the (rational) involutions $\iota$ on $F_g(X)$ and $\iota'$ on $F_g(X')$. Its inverse is $\phi_{c'}$.
\end{prop}

\begin{proof}
 Let  $\bar c$ be  a  conic  in $X$ disjoint from $c$.  The span   $\langle c,\bar c\rangle$ is a $5$-plane 
  that intersects $X\subset \P_7$ along  a canonically embedded genus-$6$ curve $c + \bar c + \Gamma_{c,\bar c}$, where  $\Gamma_{c,\bar c}$ is a sextic. Since $X\cap\langle c\rangle=c$, this implies that  $\Gamma_{c,\bar c}$ meets $c$ in four points, and similarly for 
  $\bar c$.
  
  We   now show that on $X$ general, for general conics $c$ and $\bar c$, the sextic $\Gamma_{c,\bar c}$ is smooth and irreducible. Note to that effect that  $S=G(2,V_5)\cap  \langle c,\bar c\rangle$ is a general smooth del Pezzo surface of degree $5$, which can therefore be expressed as the blow-up of the plane in four general points, with exceptional divisors $E_1,\dots,E_4$. If $h$ is the class of a general line in the plane, the embedding $S\to \langle c,\bar c\rangle$ is given by the linear system $|3h-E_1-\dots-E_4|$. The family of conics in $S$ is the union of the five pencils $|h-E_i|$, for $i\in\{1,\dots,4\}$, and $|2h-E_1-\dots-E_4|$. Since $c$ and $\bar c$ are disjoint, they must belong to the same pencil, say $|h-E_1|$ or $|2h-E_1-\dots-E_4|$.
The curve $c + \bar c + \Gamma_{c,\bar c}$ is the intersection of $S$ with a general quadric containing $c$ and $\bar c$. It follows that $ \Gamma_{c,\bar c}$ is a general member of the linear systems $|4h-2E_2-\dots-2E_4|$ or $|2h|$, which are both base-point-free. In particular, $ \Gamma_{c,\bar c}$ is a smooth rational curve of degree $6$.

 The rational map $\psi_c$ is defined on $\Gamma_{c,\bar c}$  by a linear subsystem of   $|\cI_c^4(3)|$, of  degree   
$\le 3\deg(\Gamma_{c,\bar c})-4c\cdot \Gamma_{c,\bar c}= 2$. For $\bar c$ general, its image  on $X'$ is therefore a conic.  This defines, for $X$ and $c$ general,    a rational map
$$\phi_c: F_g(X) \dra F_g(X').$$

Alternatively, a case-by-case analysis using footnote \ref{note11} shows that  for any smooth $\tau$-conic $c$,  the curve $\psi_c(\Gamma_{c,\bar c})$ is a conic as soon as
  $\bar c$ and $\iota(\bar c)$ are $\tau$-conics that meet  none of the lines that meet $c$ or   $\iota(c)$. This implies that the map $\phi_c$ is still defined in this case.
 
\medskip
The composition $\pi_{\langle c\rangle}=\phi\circ \eps^{-1}:X\dra \P^4$ is the projection from the $2$-plane $\langle c\rangle$, hence    $\pi_{\langle c'\rangle}=\phi'\circ {\eps'}^{-1}:X'\dra \P^4$   is the projection from the $2$-plane $\langle c'\rangle$. It follows that
  the curve $\psi_c(c + \bar c + \Gamma_{c,\bar c})$ in $X'$ lies in the $5$-plane $\pi_{\langle c'\rangle}^{-1}(\pi_{\langle c\rangle}(\langle c,\bar c\rangle))$. It contains the conic $\psi_c(\Gamma_{c,\bar c})$, the sextic $\psi_c(c')$, and the conic $c'$. Thus, the rational map $\phi_{c'}: F_g(X') \dra F_g(X)$ is the inverse of $\phi_c $, which is therefore birational.
  
  \medskip
The   section 
    of the quartic $\pi_c(X)\subset \P^4$ by the $2$-plane $ \pi_c(\langle c,\bar c\rangle)$ is the union of the conics $\pi_c(\bar c)$ and $\pi_c( \Gamma_{c,\bar c})$. Similarly, its section by the $2$-plane $ \pi_c(\langle c,\iota(\bar c)\rangle)$ is the union of the conics $\pi_c(\iota(\bar c))$ and $\pi_c( \Gamma_{c,\iota(\bar c)})$. Since $\bar c$ and $\iota(\bar c)$ together span a $3$-plane,  these two $2$-planes meet along a line that meets $\pi_c(X)$ in four points, including the two points of $\pi_c(\bar c)\cap \pi_c(\iota(\bar c))$. It follows that the intersection of the conics $\pi_c( \Gamma_{c,\bar c})$ and $\pi_c( \Gamma_{c,\iota(\bar c)})$ with this line must be the same. In particular, these conics meet in two points, hence so do their images $\phi_c([\bar c])$ and $\phi_c(\iota([\bar c]))$ in $X'$. This proves $\phi_c(\iota([\bar c]))=\iota'(\phi_c( [\bar c]) ) $.
\end{proof}

 \begin{prop}\label{con}
Let   $X$ be a general Fano threefold of type $\cX_{10}$ and let $c$ be a smooth $\tau$-conic   on $X$.
 The map $\phi_c:F_g(X)\dra F_g(X')$ sends 
 \begin{itemize}
\item the  curve of $\sigma$-conics on $X$ to the point  $[c']$;
\item the point $[c_X]$ to   $\iota'([c'])$; 
\item the point $\iota([c])$ to    $[c_{X'}]$; 
  \end{itemize}
and  factors as 
$$
 F_g(X) \lra F_m(X) \isomlra   F_m(X') \longleftarrow F_g(X') $$ 
  In other words, the surface $F_g(X')$ is isomorphic to the surface $F_m(X)$ blown up at the point $[c]$.
   \end{prop}

 \begin{proof} We can assume that $c$ is general.  Let $\bar c$ be a general  $\sigma$-conic on $X$.  We have
$$X\cap\langle c,\bar c\rangle=c + \bar c + \Gamma_{c,\bar c},$$ where $\Gamma_{c,\bar c}$ is a smooth irreducible sextic.  To show   $\psi_c(\Gamma_{c,\bar c})=c'$, since
 $$\chi^*E'\cdot \Gamma_{c,\bar c}=(-2\eps^*K_X-3E)\cdot  \Gamma_{c,\bar c}=0,$$
  it is   enough to show that (the strict transform of) $\Gamma_{c,\bar c}$ meets the divisor $\chi^*E'$, because it will then have to be contained in it. 

 Recall   that $c$ meets exactly two $\sigma$-conics, say $c_1$ and $c_2$. Since $c_1$ meets $c$ in one point, $\iota(c)$ in one point, is not contained in the indeterminacy locus of $\psi_c$, and has degree $-2-4+3\times 2=0$ on 
 the  linear system $|\cI_{\iota(c)}^2 \otimes\cI_c^4(3)|$ that defines $\psi_c$ (see footnote \ref{note11}), it maps to a point of $c'$ on $X'$, hence to a fiber of the map 
$\eps\vert_{E'}:E'\to c'$ on $\widetilde X'$. 
It is therefore enough to show that $\Gamma_{c,\bar c}$ meets  $c_1$ outside of $c$.

Note that $c_1$ and $c_2$ both  meet both $c$, and they both meet    $c_X$  in two points. They are therefore contained in $\langle c,c_X\rangle$, and the sextic  $\Gamma_{c,c_X}$ is the union of $c_1$, $c_2$,  and a   conic   that meets $c_1$ and $c_2$ each in  one point, and $c$ in two points, hence must be $\iota(c)$. 

 The $1$-cycle 
\begin{eqnarray*}
\Gamma&=&(c + \bar c + \Gamma_{c,\bar c})+(c + c_X + \Gamma_{c,c_X})\\
&=&2c+\bar c + \Gamma_{c,\bar c}+c_X+c_1+c_2+\iota(c)
\end{eqnarray*}
 is the complete intersection in $\P_7$ of $X$ with  the hyperplane 
 $\langle c,\bar c,c_X\rangle
 $ and the reducible quadric $\langle c,\bar c \rangle\cup
 \langle c,c_X\rangle
 $, hence $\omega_\Gamma=\cO_\Gamma(2)$. In particular, $c_1$ must meet  the other components in six points. Since it meets $c_X$ twice, $c$ and $\iota(c)$ simply, and neither $\bar c$ nor $c_2$, it must meet $\Gamma_{c,\bar c}$. 
 
 We have therefore shown that $\phi_c$ is defined at   general points of the curve $L_\sigma\subset F(X)$ of  $\sigma$-conics in $X$, and that its value at these points is the point $[c']\in F(X')$. Since $\phi_c$ commutes with $\iota$, it is also defined at $[c_X]$, where it takes the value $\iota'([c'])$.

The rest of the statements follows from the symmetry $\phi_c^{-1}=\phi_{c'}$ and the facts that $\phi_c$ commutes with the involutions  and the surfaces $F_m(X)$ and $F_m(X')$ are   of general type (Corollary \ref{min}).
 \end{proof}

     \begin{coro}\label{gene}
  If $X$ is a general   Fano threefold of type $\cX_{10}$, the sheaf $\Omega_{F_m(X)}$ is generated by its global sections.
  \end{coro}
  
  \begin{proof}
  We saw in \S\ref{stbt} that $\Omega_{F_m(X)}$ is generated by global sections, except possibly at $[c_X]$ and $\iota([c_X])$. Because of Proposition \ref{con}, this holds everywhere.
  \end{proof}

We will now write $X_c$ instead of $X'$ to highlight the dependence on $c$.    For $X$ general, as $[c]$ varies in the open subset of $F(X)$ consisting of smooth $\tau$-conics, the assignment $[c]\mapsto [X_c]$ defines a rational map $F_g(X)\dra \cN_{10}$.

   \begin{theo}\label{fibb}
Let $X$  be a general Fano threefold of type $\cX_{10}$. For {\em any} conic $c\subset X$, one can define a {\em smooth} birational model $X_c$ of $X$ such that $F_g(X_c)$ is isomorphic to the surface $F_m(X)$ blown up at the point $[c]$. The assignment $[c]\mapsto [X_c]$ defines an {\em injective  morphism} $F_{m,\iota}(X)\to \cN_{10}$ whose image is a proper surface contained in a fiber of the period map $\wp:\cN_{10}\to\cA_{10}$. 
\end{theo}

We will see in Remark \ref{mmm} that this morphism actually induces an {\em isomorphism} from  $F_{m,\iota}(X) $ onto a connected component of the fiber of the period map.

\begin{proof}
 The idea is to prove that for a {\em general} conic $d\subset X$, the conic  
$$\phi_d( c )= \psi_d(\Gamma_{d,c})\subset X_d$$
is a smooth $\tau$-conic. We can then set
$$X_c = (X_d)_{\phi_d( c )}.
$$
 Note that we define the variety $X_c$, but not a particular birational map $X\dra X_c$. 
According to Proposition \ref{con}, the surface $F(X_c)$ is then isomorphic to the surface $F_m(X)$ blown-up at the point $[c]$. By the Reconstruction Theorem \ref{rec}, the isomorphism class of the variety $X_c$ is independent of the choice of $d$. This procedure therefore defines, for each general $X$, a morphism $F_g(X)\to \cN_{10}$ which, since all $\sigma$-conics have the same images in $F_m(X)$,  factors through $F_m(X)$. 

If $c$ is a $\sigma$-conic (resp. the   $\rho$-conic), we established during the proof of Proposition \ref{con}
that $ \psi_d(\Gamma_{d,c})$ is the  smooth $\tau$-conic $d'$ (resp. $\iota'(d')$). 

So we may assume that $c=\ell\cup \ell'$ is a reducible $\tau$-conic. We    proceed by contradiction, assuming that the conic $  \psi_d(\Gamma_{d,c})\subset X_d $
is  reducible. The degree-$6$ $1$-cycle $\Gamma_{d,c}$ must then split as the sum of two
 degree-$3$ $1$-cycles $\Gamma$ and $\Gamma'$,  each of which meets $d$ in two points.

We assume that $d$ is a general $\tau$-conic; more precisely, that {\em  neither $d$ nor $\iota(d)$ meets $c$,     $\iota(c)$, or any of the (finitely many) lines contained in $X$ that meet $c$.} Furthermore, if we write  (see \S\ref{pla})
   \begin{equation}\label{ll}
   \ell=\{ [V_2] \mid \langle e_1\rangle\subset V_2\subset V_3\}\quad{\rm and}\quad
\ell'=\{ [V_2] \mid \langle e'_1\rangle\subset V_2\subset V'_3\},  
   \end{equation}
   we assume that {\em the unique hyperplane   $V_4^d\subset V_5$ such that $d\subset G(2,V_4^d)$} (see \S\ref{surF}) {\em   contains neither $e_1$ nor $e'_1$.}     Note that since $c$ is a $\tau$-conic, we have
$\ell\cap \ell' =\{ e_1\wedge e'_1\}$ and $V_4^c=V_3+V'_3$. We assume that {\em the $2$-planes $\P(V_3)$ and $\P(V'_3)$ each meet the smooth quadric surface $Q_d\subset \P(V_5)$} swept out by the lines corresponding to points of $d$ (see \S\ref{pla}) {\em transversely in two points, and no two of these four points are on a line contained in $Q_d$.}

\smallskip\noindent {\it Step 1.} {\em The curve  $\Gamma$  is irreducible and meets each of the lines $\ell$ and $\ell'$ in one point.}

  If $\Gamma$ is reducible, it     contains a line $m$. If $m$  meets $\ell$ (or $\ell'$), we have by assumption $m\cap d=\vide$, hence the residual conic must meet  $m$, and   $d$ in two points: it is therefore $\iota(d)$, which is absurd since we assumed $m\cap \iota(d)=\vide$. If $\Gamma$ is reducible, it must therefore contain a smooth conic meeting $\ell$ and $\ell'$, and this conic is $\iota(c)$. The residual line must then meet $d$ in two points, which again contradicts our assumptions.

  It follows that $\Gamma$ is irreducible. Since   a cubic curve contained in $X$ cannot be bisecant to a line contained in $X$ (because the corresponding line transform would contract the cubic and this would contradict \cite{ip}, Proposition 4.3.1; see \S\ref{el}), $\Gamma$ meets each of the lines $\ell$ and $\ell'$ in one point. 
  
  \smallskip\noindent {\it Step 2.} {\em There exists a line  $\P(W_2)\subset \P(V_5)$   that meets all lines param\-etrized by $\Gamma$. It is contained in $\P(V_4^d)$ but not in the quadric $Q_d$.}

The restriction to $\Gamma$ of the   tautological subbundle  $\cS_{2,V_5}$ is isomorphic to $\cO(-1)\oplus \cO(-2)$, or to $\cO \oplus \cO(-3)$. In the latter case, all lines parametrized by $\Gamma$ pass through a fixed point. But this cannot happen since, $d$ being a $\tau$-conic, the lines in $\P(V_5)$ corresponding to the two points of $\Gamma\cap d$ are disjoint.
We are therefore in the first case, hence $\Gamma$ can be parametrized by $\gamma:t\mapsto [w(t)\wedge v(t)]$, where $w$ is linear in $t$, and $v$ is quadratic in $t$. Take for $W_2\subset V_5$ the $2$-dimensional vector space spanned  by the $ w(t)$. More precisely, we may assume $\Gamma\cap d=\{\gamma(0),\gamma(\infty)\}$ and  
$$\gamma(t)=[ (w_0+ t w_\infty)\wedge (v_0+ t v_1+ t^2 v_\infty)],$$
with $W_2=\langle w_0,w_\infty\rangle$ and $V_4^d= \langle w_0,w_\infty,v_0,v_\infty\rangle$.
 We may further assume  
 $\Gamma\cap \ell=\{\gamma(1)\} $
and
$ \Gamma\cap \ell'=\{\gamma(t_0)\} $
for some $t_0\notin\{0,1,\infty\}$. Set 
\begin{equation}\label{para}
w=w_0+ w_\infty \quad {\rm and}\quad  w'=w_0+ t_0 w_\infty,
\end{equation}
 so that $W_2=\langle w,w'\rangle$.
 Since we assumed $e_1\notin V_4^d$, we have   $e_1\notin W_2$, hence $\gamma(1) =[ e_1\wedge w]$.  Similarly, $\gamma(t_0) =[ e'_1\wedge w']$. If $\P(W_2)$ is   contained in $Q_d$,   the point $[ w]\in \P(V_5)$ must be one of the two points of $\P(V_3)\cap Q_d$, and $[ w'] $ must be one of the two points of $\P(V'_3)\cap Q_d$. Since we assumed that none of the four lines joining these points are contained in $Q_d$, this is absurd.
 
Finally, we have 
\begin{equation}\label{v1}
 v_0+ v_1+v_\infty \in\langle e_1 ,w \rangle \quad {\rm and}\quad
  v_0+ t_0v_1+t^2_0v_\infty\in\langle e'_1 ,w' \rangle, 
  \end{equation}
 hence $v_0+t_0v_\infty$ is in the hyperplane $\langle e_1,e'_1,w,w'\rangle$.
\medskip

Conversely, let us start with $[w]\in \P( V_3\cap V_4^d)$ and $[w']\in \P( V'_3\cap V_4^d)$. Assume $w\notin V'_3$ and $w'\notin V_3$, and that the line $\langle [w],[w']\rangle$ meets $Q_d$ in two points, $[w_0]$ and $[w_\infty]$. These points determine uniquely the points $[w_0\wedge v_0]$ and  $[w_\infty\wedge v_\infty]$ of $d$, together with the point $[e_1\wedge w]$ of $\ell$, and the point $[e'_1\wedge w']$ of $\ell'$.  Note that $w_0$ and $w_\infty$ are defined only up to multiplication by   nonzero scalars. But our choice of   parametrization 
of $\Gamma$ imposes $w_0+ w_\infty\in\C w$
 (see (\ref{para})),
so these scalars must be the same, and $w_0+ t_0 w_\infty\in \C w'$, which determines 
 $t_0 $ uniquely. 

Again, $v_0$ and $v_\infty$ are defined only up to multiplication by   nonzero scalars and addition of multiples of $w_0$ and $w_\infty$ respectively. But since $v_0+t_0v_\infty$ must be in the hyperplane $\langle e_1,e'_1,w,w'\rangle$, these scalars must again be the same. These changes can be achieved by noting that
  the span of 
$ w_0+ t w_\infty$ and $v_0+ t v_1+ t^2 v_\infty$ is also the span of 
$ w_0+ t w_\infty$ and $\lambda  v_0+\mu w_0+ t (\lambda  v_1+\mu w_\infty+\nu w_0)+t^2( \lambda  v_\infty + \nu w_\infty)$.

Finally,   $v_1$ is uniquely determined by   (\ref{v1}), hence the curve $\Gamma$ is unique\-ly determined by the choice 
of $[w]$ and $[w']$. Since both vary  in a projective line, 
  $\Gamma$ belongs to a two-dimensional (irreducible)  family.

Let us now look at quadrics $\Omega$ containing $c$ and $d$. Containing a given twisted cubic $\Gamma$ as above imposes three further conditions. Since $\Gamma$ varies in a two-dimensional family, a general quadric containing $d$ and $c$ contains no such cubic $\Gamma$.  This implies what we need.


Finally, we will show in Theorem \ref{rec} that the conic transforms $X_c$ and $X_d$ are isomorphic if and only if there exists an automorphism $\sigma$ of $F_m(X)$ such that $\sigma([c])=[d]$. By Corollary \ref{auto}, $\sigma$ is either trivial or $\iota$, so this proves the last statement of the theorem.  \end{proof}

\subsection{Elementary transformation along a line}\label{el} 
Let $\ell$ be a  line contained in a threefold $X$ of type $\cX_{10}^0$. We can define an elementary transformation along $\ell$ as in \S\ref{ec}.  If $\eps:\widetilde X\to X$ is the blow-up of $\ell$, with exceptional divisor $E $, projection from $\ell$ induces a birational  morphism $\phi_{|-K_{\widetilde X}|}:\widetilde X\to \P^5$, whose image $\bar X $ has degree $10$ and has terminal hypersurface  singularities. The only curves   contracted by $\phi_{|-K_{\widetilde X}|}$ are   the strict transforms of the   lines  that meet $\ell$ and, when the normal bundle to $\ell$ in $X$ is $\cO(1)\oplus\cO(-2)$, the exceptional section of the ruled surface $E\isom \F_3$; in any event, there are finitely many such curves (\cite{ip}, Proposition 4.3.1 and Corollary 4.3.2).

Performing a flop, we end up as in \S\ref{ec} with a diagram
\begin{equation}\label{dia2}
\xymatrix
{\widetilde X\ar[dd]_{\eps}\ar[dr]_\phi\ar@{-->}[rr]^\chi&&\widetilde X'\ar[dd]^{\eps' }\ar[dl]^{\phi'}\\
&\bar X\\
X\ar@{-->}[rr]^{\psi_\ell}&&X_\ell
}\end{equation}
where   (\cite{ip}, Proposition 4.3.3.(iii)):
\begin{itemize} 
\item $X_\ell$ is again a smooth Fano threefold of degree $10$ in $\P^7$;
\item $\eps'$ is the blow-up of a line  $\ell'$ in $X_\ell$, with exceptional divisor   $E'\equiv - K_{\widetilde X'}- \chi(E) $.
\end{itemize} 
If $H'$ is a hyperplane section of $X_\ell$, we have
$$\chi^*{\eps'}^*H'\equiv \chi^*(-K_{\widetilde X'}+E')\equiv \chi^*(-2K_{\widetilde X'}-\chi(E)) 
\equiv -2\eps^*K_X-3E, 
$$
hence $\psi_\ell$ is associated with a  linear subsystem of $|\cI_\ell^3(2)|$.

As 
in \S\ref{ec},  the elementary  transformation of $X_\ell$ along the line $\ell'$ is $\psi_\ell^{-1}:X_\ell\dra X$.

\section{The   period map}

\subsection{The differential of the period map}\label{diff}

Let $X$ be a Fano threefold of type $\cX_{10}^0$.  The differential of the period map $\wp:\cN_{10}\to\cA_{10}$ at the point  defined by  $X$ is the map 
$$d\wp : H^1(X,T_X)\to \Hom(H^{1,2}(X), H^{2,1}(X))$$
defined by the natural pairing $H^1(X,T_X)\otimes H^1(X,\Omega^2_X)\to
H^2(X,\Omega^1_X)$. We want to describe the kernel of $d\wp$.

\begin{theo}\label{diffper}
 The kernel of $d\wp$  is naturally identified with the quotient
$V_5/U_3$. In particular $\wp$ is smooth and its image has dimension $20$. 
\end{theo}

\begin{proof}Following \cite{flenner}, we  use the long exact sequences induced by the 
exact sequences (\ref{conorWX}) and (\ref{resol}) to get
a commutative diagram
\begin{equation}\label{ddia}
\begin{array}{ccccc}
H^1(X,\Omega^2_W(1)\vert_X) & \otimes & H^1(X,\Omega^2_X) & \to & 
H^2(X,\Omega^1_W\vert_X) \\
\downarrow &&\Vert & &\downarrow \\
H^1(X,T_X) & \otimes & H^1(X,\Omega^2_X) & \to & H^2(X,\Omega^1_X) \\
\downarrow &&\Vert & &\downarrow  \\
H^2(X,\Omega^1_X(-1)) & \otimes & H^1(X,\Omega^2_X) & \to & 
H^3(X,\cO_X(-2)) .
\end{array}\end{equation}
If we set $G=G(2,V_5)$, the kernel of the 
restriction map 
$$\wedge^2V_5^\vee=H^0(G,\cO_G(1))\to H^0(W,\cO_W(1))$$ is 
$V_8^{\perp}$. From the normal sequences of $X$ in $ W$
and $W$ in $ G$, we deduce   isomorphisms
$$H^2(X,\Omega^1_W\vert_X)\simeq H^3(W,\Omega^1_W(-2))
\simeq H^4(W,V_8^{\perp}\ot\cO_W(-3))\simeq V_8^{\perp}$$
(because $\omega_W=\cO_W(-3)$). On the other hand,
$$H^3(X,\cO_X(-2))\isom H^0(X,\cO_X(1))^\vee\simeq H^0(W,\cO_W(1))^\vee\simeq V_8.$$ 
The rightmost column of the diagram (\ref{ddia}) is therefore   exact since the top arrow is injective 
(because $H^2(X,\cO_X(-2))=0$ by Kodaira vanishing) and 
$H^2(X,\Omega^1_X)$ is $10$-dimensional. 

We check that the pairing of the bottom row of (\ref{ddia}) is nondegenerate 
with respect to the first factor. By \cite{flenner}, Lemma 2.9, 
this follows from the fact that the following diagram 
commutes up to sign:
$$\begin{array}{ccccc}
H^2(X,\Omega^1_X(-1)) & \otimes & H^1(X,\Omega^2_X) & \to & 
H^3(X,\cO_X(-2)) \\
\downarrow && \uparrow &&\Vert \\
H^3(X,\cO_X(-3)) & \otimes & H^0(X,\cO_X(1)) & \to & 
H^3(X,\cO_X(-2)).
\end{array}$$
Indeed, the last row of this diagram is Serre-dual to the 
multiplication map $H^0(X,\cO_X(1))\ot H^0(X,\cO_X(1))\to
H^0(X,\cO_X(2))$; in particular, it is nondegenerate 
with respect to the first factor, and the same conclusion 
follows for the first row because the map 
$$H^2(X,\Omega^1_X(-1))\to H^3(X,\cO_X(-3))$$
is injective. To prove that, we need to check that $H^2(X,\Omega^1_W(-1)\vert_X)$ vanishes. But this follows from the normal sequence of $X$ in $ W$, (\ref{-1}), and (\ref{-3}).

We can now conclude that the kernel of $d\wp$ at $[X]$ is 
contained in the image of $H^1(X,\Omega^2_W(1)\vert_X)$ in 
$H^1(X,T_X)$. We compute the dimension of this space. From the normal sequence of $X$ in $ W$, we get an exact sequence 
\begin{multline*}
H^1(W,\Omega^2_W(-1))\to H^1(W,\Omega^2_W(1))\\
\to H^1(X,\Omega^2_W(1)\vert_X)\to H^2(W,\Omega^2_W(-1)) 
\end{multline*}
whose extreme terms vanish by (\ref{-1}). Moreover, $H^1(W,\Omega^2_W(1))$ is dual to
$H^3(W,\Omega^2_W(-1))$, which is $2$-dimensional   by (\ref{-1}) again.

Since we already 
know that the kernel of $d\wp$ has dimension   at least $2$, 
we can conclude that this kernel is exactly 
$H^1(X,\Omega^2_W(1)\vert_X)$. We have therefore proved the theorem.\end{proof}

\begin{rema}\label{mmm}
Let $X$ be general. Recall from Theorem \ref{fibb} that conic transforms yield an injective   morphism  from the  {\em proper} surface  $F_{m,\iota}(X)$ to the fiber $\cF$ of the period map passing through the point corresponding to $X$. By  Theorem \ref{diffper}, the image is a smooth surface  and a connected component of $\cF$. Since $F_{m,\iota}(X)$ is a minimal surface, this component is actually isomorphic to $F_{m,\iota}(X)$.
\end{rema}

\begin{rema}
 Let $X$ be general. Using iterated conic and line transformations, we can construct many threefolds with the same intermediate Jacobian as $X$. However, since we know that the fibers of the period map have dimension $2$, there must be coincidences that we will know explain.

We just saw that the connected component of the fiber $\cF$ of the period map passing through the point corresponding to $X$ is a proper surface (isomorphic to) $F_{m,\iota}(X)$. 
  If $\ell$ is a line contained in $X$, we get another component $F_{m,\iota}(X_\ell)$  of $\cF$. By continuity,  this component is independent of the choice of $\ell$ (in \S\ref{mx}, we will prove directly that these surfaces are abstractly isomorphic).
Now take a conic $c\subset X$;  by the same reasoning, conic transforms of $X_c$ also land in the component $F_{m,\iota}(X)$ of $\cF$, whereas line transforms of $X_c$ land in the ``other'' component (note that we have not proved that these two components are actually distinct).  In other words,  conic transformations leave each of the two components above of $\cF$ invariant, whereas   line transformations switch them.
\end{rema}

\subsection{Fano threefolds of degree $10$ and quartic double solids}

Quartic double solids are double covers of $\P^3$ branched along a smooth quartic surface. They have Picard number $1$, index $2$, and form the family $\cY_2$ described in the introduction. They were extensively studied by Clemens in \cite{cle}, who computed, among many other things, that their intermediate Jacobian has dimension $10$.

\begin{prop}\label{explain}
A general quartic double solid  is   birational to a  $2$-dimensional family of nodal Fano threefolds which are degenerations of (smooth) Fano threefolds of type $\cX_{10}$. 
\end{prop}

\begin{proof} This is a construction that can be found in \cite{bcz} \S4.4.1 and \cite{csp}, Example 1.11, to which we refer for the proofs.
  Let $\pi:Y\to \P^3$ be a quartic double solid, with involution $\iota$, and let $v:Y\to \P^{10}$ be the morphism defined by the anticanonical linear system $|-K_Y|=|2\pi^*H|$. {\em Choose any line $\ell\subset Y$} (there is a $2$-dimensional family of such lines).  The projection of $v(Y)\subset\P^{10}\dra \P^7$ from the $2$-plane spanned by the (conic) image of  $\ell $ is a singular Fano threefold   $\bar X \subset \P^7$,   obtained as the image  of the blow-up $\eps:\tilde Y\to Y$ of $\ell$, with exceptional divisor $E$,  by the morphism $\phi$ associated with the anticanonical linear system $|-K_{\tilde Y}|=|2\eps^*\pi^*H-E|$ (compare with diagram (\ref{dia})). This morphism only contracts   the curve $\iota(\ell)$ and $\bar X $ is a singular Fano threefold of degree $10$ with one node (from the family $(T_7)$  of \cite{csp}, Theorem 1.6).
  
  By   \cite{nam}, $\bar X$ can be deformed into a smooth Fano threefold of degree $10$, whose Picard number is $1$ by \cite{jah}; it is therefore of type $\cX_{10}$.
  \end{proof}
  
 Keeping the notation of the proof, lines on $Y$ map to conics on $\bar X $. This gives a map $F(Y)\dra F (\bar X)$. This explains the  coincidences between the numerical invariants of the minimal  surface   $F(Y) $  and those the minimal model of   $F(X_{10})$ (see footnote \ref{pgq}).

\begin{coro}\label{simple}
 The endomorphism ring of the intermediate Jacobian of a very  general Fano threefold   of type  $\cX_{10}$ is isomorphic to $\Z$.
\end{coro}

\begin{proof} 
By Proposition \ref{explain}, the intermediate Jacobian of a quartic double solid is a degeneration of  intermediate Jacobians of Fano threefolds of type $\cX_{10}$. By \cite{cle}, Theorem (5.67),   intermediate Jacobians of  quartic double solids   degenerate in turn to  semi-abelian varieties which are extensions by $\C^*$ of Jacobians of complete intersection curves $C\subset\P^3$ of bidegree $(2,4)$, where the extension class $\eps$ is the difference between the two $g^1_4$ on $C$ (\cite{cle}, end of \S7). In general, the endomorphism ring of $J(C)$ is trivial (\cite{cvdg}) and $\eps$ has infinite order. This proves the corollary.
\end{proof}

\begin{coro}\label{notiso}
 A general Fano threefold   of type  $\cX_{10}$ is {\em not} birationally isomorphic to a smooth quartic double solid.
\end{coro}

\begin{proof} 
Let $X$ be a general Fano threefold   of type  $\cX_{10}$.
  It was observed by Clemens and Griffiths (\cite{cg}; see also \cite{ip}, \S8.1) that  the Griffiths component $J_G(X)$ (the product of the principally polarized factors of $J(X)$ that are {\em not} Jacobians of curves) is a birational invariant of $X$. 
 
 It follows from the proof of Corollary \ref{simple}   that $J(X)$ can  degenerate to the intermediate Jacobian of a general quartic double solid, whose theta divisor has a singular locus of codimension $5$ (\cite{deb}, th. (8.1)). It follows that the singular locus of the theta divisor of $J(X)$ has codimension $\ge 5$, and in particular, $J_G(X)=J(X)$. 
 
  Assume now that    $X$   is birationally isomorphic to a smooth quartic double solid $Y$. We then have $J(Y)=J_G(Y)\isom J_G(X)=J(X)$. But this is impossible, since the $J(X)$ form a $20$-dimensional family (Theorem \ref{diffper}), whereas  quartic double solids  form a $19$-dimensional family.\end{proof}

\section{The moduli space $\cM_X(2;1,5)$}\label{mx}

Let $X$ be {\em any} threefold  of type $\cX_{10}^0$. We study the moduli space $\cM_X(2;1,5)$ of  semistable rank-$2$ torsion-free sheaves on $X$, with Chern numbers $c_1=1$, $c_2=5$, and $c_3=0$, 
along the lines of \cite{pf}. A noteworthy difference is that 
a vector bundle in this moduli space is not necessarily 
generated by its global sections. Our treatment of the non globally 
generated bundles is directly inspired by \cite{bf}. The main result of this section is Theorem \ref{th82}.

\medskip
\noindent {\it First remarks}. 
Let $\cE$ be a semistable rank-$2$ vector bundle on $X$, with Chern 
numbers $c_1=1$ and $c_2=5$. By the  Riemann-Roch theorem,
we have
$$\chi(X,\cE nd(\cE))=4\chi(\cO_X)-\frac{1}{2}c_1(X)(c_1(\cE)^2-4c_2(\cE))=-1.$$

We claim that $\cE$ has nonzero sections. 
Indeed $\chi(X,\cE)=4$ and $H^3(X,\cE)=0$ by
Serre duality. So we just need to prove that $H^2(X,\cE) $ vanishes, 
which can be proved as in \cite{pf}, Lemma 5.1.

Restrict $\cE$ to a general hyperplane section $S$ of $X$, a K3
surface  of degree $10$ with $\Pic(S)=\Z[\cO_S(1)]$.
We have $\chi(S,\cE\vert_S)=4$ and $H^2(S,\cE\vert_S)=H^1(S,\cE\vert_S^\vee)=0$,
hence $h^0(S,\cE\vert_S )\ge 4$. Choosing a nonzero section $s$ of 
$\cE\vert_S$, we get an extension
$$ 0 \to\cO_S \stackrel{{}\cdot s}{\to} \cE\vert_S \to\cI_Z(1)\to 0, $$
where $Z$ is a zero-dimensional scheme of length $4$. Since $\cE\vert_S$ is 
locally free, we have  $h^1(S,\cI_Z(1))\ge 1$, and there is in fact  
equality since otherwise $Z$ would generate
a line contained in $S$, which is impossible. Looking at the associated 
long exact sequence, we get
$H^1(S,\cE\vert_S)=0$, thus $h^0(S,\cE\vert_S)=4$. Back to $X$, 
we obtain $H^1(X,\cE)=0$ and $h^0(X,\cE)=4$. Moreover, 
$\cE$ is globally generated in codimension two. 

\medskip\noindent {\it Non locally free sheaves.} 
Let $\cF$ be a semistable rank-$2$ torsion-free sheaf  on $X$
with Chern numbers $c_1=1$, $c_2=5$, and $c_3=0$, which is {\em not} 
locally free. We claim that there
exists a unique line $\ell$ in $X$ such that $\cF$ fits in an exact sequence
$$0\to \cF\to T_X^\vee\to\cO_{\ell}\to 0. $$
This is the same statement as Proposition 5.11 in \cite{pf} and the proof 
is the same: the bidual $\cG=\cF^{\vee\vee}$ is locally free outside a 
finite set, so its restriction to a general
K3 section $S$ is locally free,   the quotient $ (\cG/\cF)\vert_S$ has finite 
support, and  $\len((\cG/\cF)\vert_S)=5-c_2(\cG\vert_S)$.  Moreover $\cG\vert_S$ is 
semistable and the moduli space of simple sheaves on $S$ is smooth at 
$\cG\vert_S$, of dimension $4-4\len((\cG/\cF)\vert_S)$. The same arguments as in 
\cite{pf} yields  that the length of  $ (\cG/\cF)\vert_S $ is nonzero,
hence must in fact be $ 1$. This means
that $ \cG/\cF $ is supported on a line $\ell$ plus possibly a finite set 
of points.

Moreover  $\cG\vert_S$, being semistable with Chern numbers $c_1=1$ and $c_2=4$, 
must be the restriction of $T_X^\vee$ to $S$. Finally, exactly as in 
\cite{pf}, $\cG\isom T_X^\vee$ and $\cG/\cF\isom\cO_{\ell}$.

\medskip
\noindent {\it Non globally generated bundles.}
Now, let $\cE$ be a semistable rank-$2$ vector bundle  on $X$,
with Chern numbers $c_1=1$ and $c_2=5$, which is {\em not} 
generated by global sections. We follow the arguments of \cite{bf}, 
Lemma 6.12. Denote by $\cI\subset \cE$ the image 
of the evaluation map $H^0(X,\cE)\ot \cO_X\to \cE$ and by $\cK$ its kernel. 
Set also $\cT=\cE/\cI$, a nonzero sheaf supported in codimension two. 
The torsion-free sheaf $\cI$ is stable of rank two, with $c_1(\cI)=1$ and 
$c_2(\cI)\ge 5$. The reflexive sheaf $\cK$ is also stable of rank two, 
with $c_1(\cK)=-1$. By restricting to a general  K3 section and applying 
Mukai's formula for the dimension of the moduli space, we get 
$c_2(\cK)\ge 4$, hence $c_2(\cI)=10-c_2(\cK)\le 6$. 
 
\smallskip
Suppose $c_2(\cK)=5$; then  $c(\cT)=1+c_3(\cK)$ and, by 
Riemann-Roch, $\ell(\cT)=c_3(\cK)/2\le 0$. So $\cT=0$, a 
contradiction. 

So $c_2(\cK)=4$. But the dual sheaf $\cK^\vee$ must then  be the 
dual $T_X^\vee$ of the tautological sheaf. Finally, a 
computation yields $\chi(\cT(k))=k$. Since $H^0(X,\cT)=0$, 
this implies   $\cT=\cO_\ell(-1)$ for some line $\ell$
in $X$. We get an exact sequence
$$0\to T_X\to H^0(X,\cE)\ot \cO_X\to \cE\to \cO_\ell(-1)\to 0.$$
Note that if we apply the functor ${\mathcal Hom}(.,\cO_X)$, we get 
the dual sequence
$$0\to \cE^\vee\to H^0(X,\cE)^\vee\ot \cO_X\to T_X^\vee\to \cO_\ell\to 0.$$
This means that $\cE^\vee$ can be identified with the kernel of the 
evaluation map $H^0(X,\cF)\ot \cO_X\to \cF$ of the non locally free
sheaf $[\cF]\in \cM_X(2;1,5)$ defined by $\ell$. In particular, 
$\cE$ is uniquely defined by $\ell$, and $H^0(X,\cE)\simeq H^0(X,\cF)^\vee$. 

\medskip
\noindent {\it An involution on the moduli space.}
Let $[\cE]\in \cM_X(2;1,5)$ be a globally generated vector bundle. 
Observe that the kernel $\cK$ of the evaluation map $H^0(X,\cE)\ot\cO_X\to \cE$
is a rank-$2$ vector bundle with $c_1 =-1$ and $c_2 =5$, with no nonzero
global sections. Its dual $\iota \cE $ is therefore stable, 
with Chern numbers $c_1(\iota \cE)=1$ and $c_2(\iota \cE)=5$. Moreover, 
$\iota \cE$ is globally generated and we have a natural identification 
$H^0(X,\iota \cE)\simeq H^0(X,\cE)^\vee$.

More precisely, taking global sections in the sequence 
above, we get an exact sequence of vector spaces
$$0\to H^0(X,\cE)^\vee\to H^0(X,T_X^\vee)\simeq V_5^\vee\to V_1^{\vee}\to 0,$$
where   $V_1\subset V_5$ is the ``vertex'' of the line $\ell\subset X$ (\S\ref{pla}). 
We can thus identify $H^0(X,\cE)$ with $V_5/V_1$, hence $\P (H^0(X,\cE))$ 
with the set of $V_2\subset V_5$ containing $V_1$. Geometrically, 
the codimension-$2$ Schubert cycle $\sigma_{20}(V_2)\subset G(2,V_5)$
meets $X$ along a $1$-cycle of the form $\ell+\Gamma^1_5$,
where $\Gamma^1_5$ is an elliptic quintic bisecant to $\ell$. This cycle 
is the zero-locus of the section of $\cE$ defined (up to scalars)
by $V_2$.

\begin{prop}\label{pro81}
Let $X$ be a Fano threefold of type $\cX_{10}^0$ and let $\ell\subset X$ be a line. There is a birational isomorphism 
$$\phi_\ell:\cM_X(2;1,5)  \dra F (X_\ell)$$ which
is compatible with the   involutions on $\cM_X(2;1,5)$
and $F (X_\ell)$.
\end{prop}

\begin{proof} Let $[\cE]\in \cM_X(2;1,5)$ be a globally generated vector bundle and consider the exact sequence
$$0\to\cI_\ell\ot \cE\to \cE\to \cE\vert_\ell\simeq\cO_\ell\oplus
\cO_\ell(1)\to 0.$$
Since 
every proper subspace 
of $H^0(\ell,\cO_\ell(1))$ has a base-point, and
$\cE$ is globally generated, the   
restriction map $H^0(X,\cE)\to H^0(\ell,\cE\vert_\ell)$ is surjective. Hence
$H^0(X,\cI_\ell\ot \cE) $ is one-dimensional, generated by  $s$. The zero-locus of
$s$ is $\Gamma^0_4\cup\ell$, where $\Gamma^0_4$ is a rational quartic curve bisecant 
to $\ell$. The line transform $\psi_\ell$ maps $\Gamma^0_4$ to a conic 
on $X_{\ell}$.  This defines a rational map
$\phi_\ell$ as required.

Conversely,  the inverse image by $\psi_\ell $ of a general conic on 
$X_\ell$ is a rational quartic $\Gamma^0_4$ bisecant to $\ell$. Applying the Serre 
construction to the degenerate elliptic quintic $\Gamma^0_4\cup\ell$,
we obtain a stable vector bundle of rank $2$ on $X$ with Chern
numbers $ 1$ and $ 5$ and a section that vanishes on $\Gamma^0_4\cup\ell$.
The map $\phi_\ell$ is therefore birational. We know prove that it
is compatible with the involutions.

 Recall that $\iota \cE$ is dual to the kernel of the evaluation 
map $H^0(X,\cE)\ot\cO_X\to \cE$. Let $s$ (resp. $s'$) be a section of $\cE$ 
(resp. $\iota \cE$) generating 
$H^0( X,\cI_\ell\ot \cE)$ (resp. $H^0(X,\cI_\ell\ot \iota \cE)$), 
and denote its zero-locus by $\Gamma^0_4\cup \ell$ (resp. ${\Gamma'}^0_4\cup\ell$). 
By the definition of $\iota \cE$, the section $s$ defines on $\Gamma^0_4$ a 
section $\sigma$ of $(\iota \cE)^{\vee}$, and this section does not vanish
(since it defines a nonzero section of $ \cE \vert_{\Gamma^0_4} $). We get an exact 
sequence 
$$
0\to\cO_{\Gamma^0_4}\to (\iota \cE)^{\vee}\vert_{\Gamma^0_4}\to \cL^{\vee}\to 0,$$
where $\cL$ is a line bundle on ${\Gamma^0_4}$ of degree $4$. On the other hand, 
the pairing $(\sigma,s'\vert_{\Gamma^0_4})\in\cO_{\Gamma^0_4}$ is identically zero, 
since $s'$ vanishes on $\ell$, which meets ${\Gamma^0_4}$. This means that 
on ${\Gamma^0_4}$, the section $s'$ of $\iota \cE$ is in fact a section 
of $\cL$. Therefore $s'$ must vanish at four points of ${\Gamma^0_4}$:
the two intersection points with $\ell$, and two other points 
which must be on ${\Gamma'}^0_4$. Thus ${\Gamma^0_4} $ and ${\Gamma'}^0_4$ meet in two points, 
hence the corresponding conics $c=\psi_\ell({\Gamma^0_4})$ and 
$c'=\psi_\ell({\Gamma'}^0_4)$ also meet in two points. By Lemma \ref{l51}, 
these two conics are in involution. 
\end{proof}

\begin{theo}\label{th82}
Let $X$ be any Fano threefold of type $\cX^0_{10}$.
The moduli space $\cM_X(2;1,5)$ of semistable rank-$2$ torsion-free sheaves on $X$ 
with Chern numbers $c_1=1$, $c_2=5$, and $c_3=0$, is a smooth irreducible 
surface. 

The non locally free sheaves and the non globally generated vector 
bundles  in $\cM_X(2;1,5)$ are respectively parametrized by two copies 
 $\Gamma_{\rm nf}$ and $\Gamma_{\rm ng}$ of the  
curve $\Gamma(X)$ of lines in $X$. These two curves are exchanged by  
the involution $\iota$ on the moduli space.  
\end{theo}

When $X$ is {\em general,} $\Gamma(X)$ is a smooth irreducible curve of genus $71$ (\cite{ip}, Theorem 4.2.7).

\begin{proof}
First suppose that $[\cE]\in \cM_X(2;1,5)$ is a globally generated 
vector bundle. A general section of $\cE$ vanishes along a smooth
elliptic quintic ${\Gamma^1_5}$ and yields an exact sequence 
$$0\to\cO_X\to \cE\to\cI_{\Gamma^1_5}(1)\to 0.$$
Twisting by $\cE^\vee$ and taking cohomology, we obtain
 $h^2(X,\cE nd(\cE))=h^2(X,\cI_{\Gamma^1_5}\otimes \cE)=h^1({\Gamma^1_5},\cE\vert_{\Gamma^1_5})$.
By the Atiyah classification, $\cE\vert_{\Gamma^1_5}$ is the direct sum of two line 
bundles, both globally generated.

Suppose that $[\cE]$ is {\em not}    a smooth point of $\cM_X(2;1,5)$. 
We have    
 $H^1({\Gamma^1_5},\cE\vert_{\Gamma^1_5})\ne 0$, and one of these line bundles is trivial: 
$\cE\vert_{\Gamma^1_5}=\cO_{\Gamma^1_5}\oplus\cO_{\Gamma^1_5}(1)$. But
$\cE\vert_{\Gamma^1_5}$ is the normal bundle to ${\Gamma^1_5}$ in $X$. In the normal sequence
$$\begin{matrix}
0&\to& N_{{\Gamma^1_5}/X}&\to &N_{{\Gamma^1_5}/G(2,V_5)}&\to &(N_{X/G(2,V_5)})\vert_{\Gamma^1_5}&\to &0 \\
&&\Vert&&\Vert&&\Vert\\
0&\to&\cO_{\Gamma^1_5}\oplus\cO_{\Gamma^1_5}(1)&\to&\cO_{\Gamma^1_5}(1)^5&\to&
\cO_{\Gamma^1_5}(1)^2 \oplus\cO_{\Gamma^1_5}(2)&\to &0
\end{matrix}$$
the image in  $N_{{\Gamma^1_5}/G(2,V_5)}$ of the factor $\cO_{\Gamma^1_5}(1)$ of $N_{{\Gamma^1_5}/X}$  
  maps to zero in $(N_{X/G(2,V_5)})\vert_{\Gamma^1_5}$, which   means that 
the corresponding linear form defines a tangent hyperplane to $X$. 
This   contradicts the smoothness of ${\Gamma^1_5}$.

\medskip
Suppose now that $[\cF]\in\cM_X(2;1,5)$ is not locally free. We saw that $\cF$ fits into an exact sequence $$0\to \cF\to T_X^\vee
\to\cO_{\ell}\to 0$$ for a unique line $\ell$ in $X$. 
We are thus in the same situation as in \cite{pf}, Proposition 5.12, 
and the same proof yields $\Ext^2(\cF,\cF)=0$. 

Finally, suppose that $\cE$ is locally free but not globally generated. 
We saw that $\cE$ fits into an exact sequence
$$0\to \cE^\vee\to H^0(X,\cF)\ot\cO_X\to \cF\to 0$$ where the non locally 
free sheaf $[\cF]\in\cM_X(2;1,5)$ is as above. From this sequence 
we infer   $$H^2(X, \cE nd(\cE))\simeq H^1(X, \cE\ot \cF)$$ and $$H^1(X, \cE\ot \cF)\simeq \Ext^2(\cF,\cF)$$ 
(apply  the functor $\Hom(\cdot,\cF)$ and use 
$H^1(X,\cF)=H^2(X,\cF)=0$). 
So by the previous case we obtain the vanishing of $H^2(X, \cE nd(\cE))$. 

\smallskip

We can now conclude that $\cM_X(2;1,5)$ is a smooth surface. Since it is birational to $F(X_\ell)$ (Proposition \ref{pro81}), it is also irreducible for $X$ {\em general.} But the vanishing of $\Ext^2(\cF,\cF)$, for any point $[\cF]$ of $\cM_X(2;1,5)$, also proves that for any family $\cX\to S$
 of Fano threefolds of type $\cX_{10}^0$, the relative moduli scheme $\cM_\cX(2;1,5)\to S$ is smooth over $S$. In particular, since its general fibers are connected, {\em any} fiber is connected. This proves that 
$\cM_X(2;1,5)$ is always connected.

\smallskip
   
The involution $\iota$ was defined by mapping a globally
generated vector bundle $\cE$ to the globally generated vector 
bundle $\iota \cE$ defined by the exact sequence 
$$0\to (\iota \cE)^\vee
\to H^0(X,\cE) \ot\cO_X\to\cE\to 0$$
 If we replace $\cE$ by the non locally
free sheaf $\cF\in\cM_X(2;1,5)$ defined by the line $\ell$, 
this exact sequence gives for $\iota \cF$ the non globally  generated 
vector bundle associated to the same line. This shows that $\iota$
extends to the curve $\Gamma_{\rm nf}$ of non locally free sheaves in 
$\cM_X(2;1,5)$, which is mapped bijectively to the curve $\Gamma_{\rm ng}$ 
of non globally generated vector bundles. Being an involution, $\iota$ also
extends to $\Gamma_{\rm ng}$. We can finally conclude that $\iota$ defines
a regular involution of the smooth surface $\cM_X(2;1,5)$ which permutes
the curves $\Gamma_{\rm nf}$ and $\Gamma_{\rm ng}$.
\end{proof}

\begin{coro}\label{cor83}
Let $X$ be any Fano threefold of type $\cX^0_{10}$. The surface $F(X)$ is irreducible.\end{coro}

\begin{proof}
This follows directly from Theorem \ref{th82}, Proposition \ref{pro81}, and the fact that $X$ is itself a line transform (\S\ref{el}).\end{proof}

 \section{The reconstruction theorem}\label{app11}
 
 Following  Logachev (\cite{lo}), we show that a general $X_{10}$ can be recovered from its Fano surface.

\begin{theo}[Logachev] \label{rec}
Let $X$ and $X'$ be  general   Fano threefolds of type $\cX_{10}$  and let $f:F_g(X)\isomto    F_g(X')$
be an isomorphism. There exists an isomorphism $\phi: X\isomto   X'$ inducing $f$.\end{theo}

\begin{coro}\label{cic}  
Assume that $X$ is general. For any smooth $\tau$-conic $c\subset X$, the threefolds  $X_c$ and $X_{\iota(c)}$ are isomorphic.
\end{coro}

 \begin{proof}This follows from Theorem \ref{rec} and Proposition \ref{con}.
  \end{proof}

\begin{coro}\label{auto}  
For a general $X$, the only nontrivial automorphism of the minimal surface $F_m(X)$ is the involution $\iota$. 
\end{coro}

 \begin{proof}
  The cotangent bundle  of   $F_m(X)$ is   generated by global sections (Corollary \ref{gene})
and defines a morphism
$$\ell_{10}: F_m(X)\to G(2,V_{10}),$$
where  $V_{10}=H^0(F_m(X),\Omega_{F_m(X)})^\vee$. Since global sections of $\Omega^1_{F_m(X)}$ are anti-invariant by the involution $\iota$ (\cite{lo}, Proposition 0.5), $\ell_{10}$ factors through   the action of $\iota$.
Let $V_8$ be the 
quotient of $V_{10}$ by the $2$-dimensional vector space $\ell_{10}([c_X])$. By the tangent bundle theorem (\S\ref{stbt}), there is a diagram
\begin{equation}\label{l10}
\xymatrix
{  F(X)\ar[dd]\ar[dr] \ar[rr]^{\ell_8\hspace*{1cm}}&&G(2,V_8)\\
  & F_\iota(X)\ar[dd]\ar[ur]_{\ell_{8,\iota}}\\
   F_m(X)\ar '[r]^(0.7){\ell_{10}} [rr]\ar[dr] && G(2,V_{10})\ar@{-->}[uu] \\
 & F_{m,\iota}(X)\ar[ur]_{\ell_{10,\iota}}&&
}\end{equation}
where the rational map $G(2,V_{10})\dashrightarrow G(2,V_8)$ is induced by the 
projection $V_{10}\to V_8$. The   map $\ell_{10,\iota}$ is generically injective, because  
$\ell_{8,\iota}$ is: $\ell_8$ sends $(c,V_4)$ to $ \langle c\cap\iota(c)\rangle$, the lines in $\P(V_5)$ corresponding to the points of the line $\langle c\cap\iota(c)\rangle$ span the hyperplane $\P(V_4)$, and we recover $c\cup\iota(c)$ by intersecting $G(2,V_4 ) $ with $ X$.

 Assume now that $X$ is very general, so that  the endomorphism ring of $  \Alb(F_m(X))\isom J(X)$ is   trivial (Corollary \ref{simple}). 
  Any nontrivial automorphism $\sigma$ of $F_m(X)$ then acts on the tangent space  $T_{J(X),0}=V_{10} $  as $\pm\Id$, hence $\ell_{10}$ factors through the action of $\sigma$. This implies $\sigma=\iota$. 
   \end{proof}

 \begin{proof}[Proof of Theorem \ref{rec}] We will reconstruct $X$ from the abstract surface $F_g(X)$. The successive steps are the following. 

1) From the abstract surface
$F_g(X)$, we first recover the minimal surface $F_m(X)$ with its involution $\iota$.  
By   Corollary \ref{gene}, we also
recover the map $F_\iota(X)\to G(2,V_8)$ which, when we will have constructed $X$, will map a  conic $c$ to 
the line  $\langle c\cap\iota(c)\rangle$. 

2) In $\P (V_8)$, we recover the plane $\Pi$ and 
the   conics $c_X$ and $c_U^\vee$. 

3) We consider the projection to $\P(V_8)\dra\P(V_8/U^\vee)=\P^4$ which,  when we will have constructed $X$, will induce
  a  birational isomorphism $\kappa:W\dra\P^4$ (Proposition \ref{kappa}). The problem at that point is that we do not know
where to locate $W$ in $\P (V_8)$. Nevertheless, we can reconstruct the rational normal cubic $\Gamma^0_3$
in $\P^4$ which defines $\kappa^{-1}$ (Proposition \ref{kappa}). 

4) The map $\kappa^{-1}$ sends $\P^4$ 
to a copy of $W$ in some $\P (V'_8)$. From $F_\iota(X)$, we are able to find an identification between $\P (V_8)$ and 
$\P (V'_8)$.  

5) So we get $W$ inside $\P (V_8)$, and intersecting it with the lines parametrized 
by $F_\iota(X)$, we get a surface $S$ whose ``quadratic span'' is the Fano threefold $X$
we are looking for.

\smallskip\noindent {\it Step 1.} We know that  $L_\sigma$ and  $L'_\sigma$ are the unique rational curves in $F_g(X)$ 
and $F_g(X')$, so $f(L_\sigma)=L'_\sigma$. Moreover  these curves are $(-1)$-curves, and contracting them, we get the 
minimal surfaces $F_m(X)$ and $F_m(X')$, and an isomorphism
$f:F_m(X)\isomto    F_m(X')$ which sends $\iota([c_X])$ to $\iota'([c_{X'}])$.
 
We keep the notation and results of the proof of Corollary \ref{auto}. We have a commutative diagram  
\begin{equation*}
\xymatrix
{ F_m(X)\ar[d]^{f}\ar[r]^{\ell_{10}} &G(2,V_{10})\ar[d]^{f} \\
F_m(X')\ar[r]^{\ell'_{10}} &G(2,V'_{10}),
}\end{equation*}
where  the maps $\ell_{10}$ and $\ell'_{10}$ have degree $2$ onto their images, and factor through the 
involutions $\iota$ and $\iota'$ respectively.
 Therefore we recover  $\iota$   and $\iota'$, and also the   
points  $[c_X]\in F_g(X)$ and $[c_{X'}]\in F_g(X')$. 

The map $f$ induces a map from the diagram (\ref{l10}) to the corresponding diagram for $X'$.

\smallskip
We must think of the map $\ell_8$ as sending a conic $c$ to the line 
$\langle c\cap\iota(c)\rangle$ -- except 
that  we do not know yet how to identify a point  in the abstract surface
$F_g(X)$ with a conic on $X$. What we have is the abstract $2$-dimensional family of lines 
$\Im(\ell_8)$ in $\P (V_8)$.
 If we knew how to construct $W$ in $\P (V_8)$, intersecting it with these lines, we would get a surface $S$, and 
the  base-locus of the quadrics containing $S$ would be $X$. 

The problem at this point is that we do not know 
how to locate $W$  in our abstract $\P (V_8)$. Logachev's idea is to use the birational
isomorphism $\kappa^{-1}:\P^4\dra W$ defined in \S\ref{birw} to reconstruct $W$.  Since its inverse $\kappa $ 
  is just the projection from the $2$-plane $\Pi=\langle c_X\rangle$, we first need to reconstruct this plane.

\smallskip\noindent {\it Step 2.} Let $L_\iota$ be the image of $L_\sigma$ in $F_\iota(X)$. 
For $[c]\in L_\iota $, the line $\ell_8([c])$ is tangent 
to the conic $c_U^\vee\subset\Pi$. So we recover the plane $\Pi$, and the
isomorphism $L_\sigma\simeq c_U^\vee\simeq c_U$. Moreover $f$ induces an isomorphism 
between $\Pi$ and the corresponding plane $\Pi'\subset\P (V_8')$, restricting to an 
isomorphism between the conics $c_U^\vee$ and $(c'_U)^\vee$. 

On the other hand, there exists a finite set of points $[c]$ in $ F_\iota(X)\moins L_\iota$ such that 
$\ell_8([c])$ meets $\Pi$. Moreover the intersection points with $\Pi$ belong to a unique conic in 
$\Pi$,\footnote{One must check that we get at least five such intersection points; see \cite{lo}.} 
namely $c_X\subset \Pi$. We can therefore recover this conic from our data, in such a way that 
$f$ induces an isomorphism between $c_X\subset \Pi$ and $c'_X\subset \Pi'$. 

\smallskip\noindent {\it Step 3.} Now  consider the projection from $V_8$ to 
$V_8/U^\vee_3$. The next claim is that we can reconstruct the rational cubic $\Gamma^0_3\subset 
\P (V_8/U^\vee_3)$. 

For this we observe that from the two  conics $c_X$ and 
$c_U^\vee$ in $\Pi$, we can define the elliptic curve 
$$\Gamma^1_X=\{(t,x)\in c_U^\vee\times c_X\mid x\in T_{c_U^\vee,t}\},$$ 
a double cover of both $c_X$ and $c_U^\vee$. Let ${\cS_{2,V_8}}$ denote the rank-$2$ tautological bundle 
on $G(2,V_8)$. There is an obvious map from $\P ({\cS_{2,V_8}})$ to $\P (V_8)$. On the other hand, 
we have a chain of inclusions 
$$\Gamma^1_X\hookrightarrow\P(\ell_8^*{\cS_{2,V_8}})\vert_{L_\iota}\hookrightarrow\P(\ell_8^*{\cS_{2,V_8}}).$$ 
The  rational map defined as the composition  
$$\P(\ell_8^*{\cS_{2,V_8}})\to \P( {\cS_{2,V_8}})\to \P (V_8)\dashrightarrow \P (V_8/U^\vee_3)$$ 
is not defined on $\Gamma^1_X$. 
However, since $\P(\ell_8^*{\cS_{2,V_8}})\vert_{L_\iota}$ is a divisor in $\P(\ell_8^*{\cS_{2,V_8}})$, the restriction to 
this divisor is a well-defined rational map. And   since $\Gamma^1_X$ is itself a divisor in 
$\P(\ell_8^*{\cS_{2,V_8}})\vert_{\ell}$, the  restriction is again well-defined. We end up with a well-defined 
map $\Gamma^1_X\to\P (V_8/U^\vee_3)$. 

Of course we can do the same for $X'$, and we get a compatible isomorphism between $\Gamma^1_X$ and $\Gamma^1_{X'}$.

Now recall (\S\ref{birw})  that the total transform of a point $x$ of $\Pi$   by the birational isomorphism $\kappa :W\dashrightarrow\P(V_8/U_3)$
  is a line $\ell_x$ that meets $\Gamma^0_3$ at the 
points corresponding to those two points in $c_U^\vee$ whose tangents pass through $x$.
This implies that a point $(t,x)$
in $\Gamma^1_X$ must be mapped in $\P (V_8/U_3)$ to some point $p(t,x)$ of the bisecant $\ell_t$. 

But there is another point $(t',x)$ on $\Gamma^1_X$,
mapping to a point $p(t',x)$ on the same bisecant line $\ell_x$. Since, in general,  $p(t,x)\ne  p(t',x)$,\footnote{This can be checked by a lengthy direct calculation.}
 we can   recover the line $\ell_x$ as the line joining 
these two points. Finally, if we consider the points  $(t,x)$ and $(t,x')$ of  $\Gamma^1_X$, we see that the lines $\ell_x$ and $\ell_{x'}$ must meet on the point of $\Gamma^0_3$ corresponding to $t$. 
So this process allows to reconstruct the curve  $\Gamma^0_3$, and an isomorphism with $c_U^\vee$.

\smallskip\noindent {\it Step 4.} From  $\Gamma^0_3\subset  \P (V_8/U_3)$, we can 
reconstruct a copy of $W$, say $W^\dagger\subset V_8^\dagger$, as the image of the birational map defined 
by the linear system of quadrics through $\Gamma^0_3$. We get the diagram
$$\begin{array}{ccccc}
\P (V_8) & \dashrightarrow &\P (V_8/U_3) &\stackrel{|\cI_{\Gamma^0_3}(2)|}{\dashrightarrow} & W^\dagger \\
\cup & & \cup & & \cap \\
W & & \Gamma^0_3 & & \P (V_8^\dagger) 
\end{array}$$
We do not know yet where to locate $W$ in $\P (V_8)$, but we know that the identification of 
$W$ with $W^\dagger$ induces a unique linear isomorphism between $\P (V_8)$ and $\P (V_8^\dagger)$. 
What we need to do is to recover this isomorphism from our data.  

First observe that to a general point $[c]$ of $F(X)$ should correspond a conic $(c,V_4)$, and
the point $\theta([c])=[U_3\cap V_4]$ in $\Pi=G(2,U_3)$. 

From our present knowledge, this point can be recovered as follows: the line 
$\ell_5([c])$, image of $\ell_8([c])$ in $G(2,V_8/U_3)$, meets the linear span of $\Gamma^0_3$ 
at one point. This point belongs to a unique bisecant line to $\Gamma^0_3$, meeting $\Gamma^0_3$ at two points 
which we can identify 
with two points of $c_U^\vee$. The intersection of the corresponding tangents is $\theta([c])$. 
Moreover, we know that the projection $\P (V_8)\dashrightarrow \P (V_8/U_3)$, once restricted to  $\P(M_{V_4}  )$, can be identified with the projection through that point $\theta([c])$; 
in particular this applies to $\ell_8([c])$, and we can conclude that $\ell_5([c])$ is the image of 
$\ell_8([c])$ by projection from $\theta([c])$.   

Now Logachev considers those pairs $(c,c')$ such that $\ell_5([c])$ and $\ell_5([c'])$ 
meet at one point, say $m(c,c')$; there is a two-dimensional family of such pairs. 
We know the image of $m(c,c')$ in $\P (V_8^\dagger)$, and we
want to reconstruct its preimage $n(c,c')$ in $\P (V_8)$. But this is easy: $\ell_5([c])$ and $\theta([c])$ 
generate a plane $\Pi_c$, which meets the corresponding plane $\Pi_{c'}$ at a unique
point, which is $n(c,c')$.\footnote{To see this, observe that $\Pi_c$ is contained in the $3$-plane 
$\Theta_c=\P(V_8)\cap \P(\wedge^2V_4)$, which meets $\Theta_{c'}$ along 
$\P(V_8)\cap \P(\wedge^2(V_4\cap V'_4))$; but in general this is just a point. So it must be $n(c,c')$.} 

We are thus able to  reconstruct the isomorphism $\P (V_8)\to\P (V_8^\dagger)$ on the special points 
$n(c,c')$. Since they are not contained in any hyperplane, they completely determine the isomorphism we were looking for.

\smallskip\noindent {\it Step 5.} In the preceding steps, we have reconstructed $W$ in $\P (V_8)$. 
Moreover, we have done that from purely projective constructions in terms of the family of lines $\ell_8([c])$,
when $[c]$ describes the abstract surface $F_m(X)$. This implies that the isomorphism $f:F_g(X)\isomto F_g(X')$
induces an isomorphism $\varphi$ between $\P (V_8)$ and $\P (V'_8)$, mapping $W$ to an isomorphic copy $W'$, and 
compatible with $\ell_8$ and $\ell'_8$. This means that we can recover a surface $S$ in $X$
by intersecting the lines $\ell_8(c)$ with $W$, and that $f$   maps $S$ to the corresponding surface 
$S'$ in $X'$. 

But $S$ determines $X$, as the intersection of the quadrics in $\P (V_8)$ containing 
$S$\footnote{Indeed, Logachev shows that the surface $S$ is cut out on $X$ by a 
hypersurface of degree $21$.} and this implies that $\varphi$ maps $X$ isomorphically  onto $X'$.
 
Finally, $u=f^{-1}\circ\varphi^*$ is an automorphism of $F_g(X)$ which descends to an automorphism $u_m$ of $F_m(X)$ 
such that $\ell_8\circ u_m=\ell_8$. Since $\ell_8$ is generically injective,    $u_m$ is the 
identity, hence so is $u$. The theorem is proved.  
\end{proof}

\end{document}